\documentclass[12pt, reqno]{amsart}

\usepackage{amsmath, amssymb, amscd, newlfont, bbm}
\usepackage{bbold}
\usepackage{enumerate}
\usepackage[foot]{amsaddr}
\usepackage[final]{hyperref}
\hypersetup{
	colorlinks=true,        
	linkcolor={black},  
	citecolor={black},
	urlcolor={black}     
}

\newcommand{\F}{{\mathbb{F}}}

\newcommand{\C}{{\mathbb{C}}}

\newcommand{\Hom}{{\mathrm{Hom}}}

\newcommand{\Irr}{{\mathrm{Irr}}}

\newcommand{\repspan}{{\mathrm{span}}}

\newcommand{\Ind}{{\mathrm{Ind}}}

\newcommand{\GL}{{\textnormal{GL}}}

\newcommand{\SL}{{\mathrm{SL}}}

\newcommand{\Sp}{{\mathrm{Sp}}}

\newcommand{\SO}{{\mathrm{SO}}}

\newcommand{\Ort}{{\mathrm{O}}}

\usepackage[mathscr]{eucal}

\numberwithin{equation}{section}

\newtheorem{Prop}[equation]{Proposition}
\newtheorem{Lem}[equation]{Lemma}

\newtheorem{Thm}[equation]{Theorem}
\newtheorem{Cor}[equation]{Corollary}
\newtheorem{Rem}[equation]{Remark}

\title
   [ Theta lifts of generic representations ]
   { Theta lifts of generic representations \\ for dual pairs $(\mathrm{Sp}_{2n}, \mathrm{O}(V))$}
\author{Petar Baki\' {c}}

\address{ Department of Mathematics,
Faculty of Science,
University of Zagreb,
Bijeni\v cka 30,
10000 Zagreb,
Croatia}
\email{pbakic@math.hr}

\subjclass[2010]{Primary 22E35 and 22E50; Secondary 11F70}
\keywords{}
\thanks{This work has been supported in part by the Croatian Science Foundation under the project IP-2018-01-3628.}

\begin{document}

\begin{abstract}
We determine the occurrence and explicitly describe the theta lifts on all levels of all the irreducible generic representations for the dual pair of groups $(\mathrm{Sp}_{2n}, \mathrm{O}(V))$ defined over a local nonarchimedean field $\F$ of characteristic $0$. As a direct application of our results, we are able to produce a series of non-generic unitarizable representations of these groups.
\end{abstract} 

\maketitle

\section{Introduction}
\label{sec:intro}

In this paper we describe the theta correspondence for the dual pair $(\Sp(W),\Ort(V))$ defined over a nonarchimedean local field $\F$ of characteristic $0$. Our main results provide a complete and explicit description of all the theta lifts of generic representations.

For $\F$ as above, we consider the usual towers of symplectic and quadratic spaces. More specifically, for $\epsilon = \pm 1$ we have $W_n =  \text{a }(-\epsilon)\text{-Hermitian space of dimension }n$, and $V_m = \text{an }\epsilon\text{-Hermitian space of dimension }m$ (see \S\ref{subs:groups}). If we denote by $G(W_n)$ and $H(V_m)$ the corresponding isometry groups, then $G(W_n) \times H(V_m)$ is a reductive dual pair inside a larger symplectic group $\Sp(W_n \otimes V_m)$. Fixing an additive character $\psi$ of $\F$, we obtain a Weil representation $\omega_{m,n}$ of the dual pair $G(W_n) \times H(V_m)$ (or the corresponding double covers). For an irreducible smooth representation $\pi$ of $G(W_n)$, the maximal $\pi$-isotypic quotient of $\omega_{m,n}$ is of the form
\[
\pi \otimes \Theta(\pi, V_m)
\]
where $\Theta(\pi, V_m)$ is an admissible representation of $H(V_m)$. The Howe duality conjecture (see Theorem \ref{thm:Howeduality}) asserts that $\Theta(\pi,V_m)$ has a unique irreducible quotient, denoted by $\theta(\pi,V_m)$, whenever it is non-zero. The basic problems regarding this construction are determining whether $\Theta(\pi,V_m)$ is non-zero and providing an explicit description of $\theta(\pi,V_m)$.

These questions have been studied by a number of authors. Important results were first obtained by Howe \cite{howe1979beyond}, Kudla \cite{kudla1986local}, Kudla-Rallis \cite{kudla2005first}, Waldspurger \cite{waldspurger1990}, and others. This paper relies mainly on the works of Mui\'{c} (\cite{muic2004howe}, \cite{muic2008theta}) which provide a complete description of $\theta(\pi,V_m)$ when $\pi$ is in discrete series, and on the more recent work of Atobe and Gan \cite{atobe2017local}, which gives an analogous description (in a somewhat broader setting) for tempered $\pi$, using $L$-parameters.

In this paper, we provide a complete answer to these questions when $\pi$ is a generic representation. We restrict ourselves to the case when both $m$ and $n$ are even, resulting in the dual pair $(\Sp(W_n), \Ort(V_m))$.

The question of determining whether $\Theta(\pi,V_m)$ is non-zero is answered in terms of the first occurrence in a Witt tower. Namely, if $(V_m)$ is a Witt tower of $\epsilon$-Hermitian spaces, then Kudla's persistence principle (see Proposition \ref{prop:towers}) guarantees that the number $m(\pi) = \min\{m : \Theta(\pi,V_m)\neq 0\}$ is finite and that $\Theta(\pi, V_{m(\pi)+2r}) \neq 0$ for any $r \geqslant 0$. The following result (Theorem \ref{thm:first_occ}) describes the first occurrence index $m(\pi)$ for any irreducible representation which is isomorphic to its standard module; this includes generic representations because of the standard module conjecture (see \S \ref{subs:generic}).
\begin{Thm}
Let $\pi \in \Irr(G(W_n))$ be a representation isomorphic to its standard module,
\[
\pi \cong \chi_V\delta_r\nu^{s_r} \times \dotsm \times \chi_V\delta_1\nu^{s_1} \rtimes \pi_0.
\]
Then $m(\pi) = m(\pi_0) + n - n_0$, where $n_0$ is defined by $\pi_0 \in \Irr(G(W_{n_0}))$.
\end{Thm}
The notation used here for parabolic induction is introduced in Section \ref{subs:reps}. Furthermore, $\chi_V$ (and analogously $\chi_W$) is the quadratic character attached to the tower $(V_m)$ (resp.~$(W_n)$); see Section \S \ref{subs:Witt}. Since the number $m(\pi_0)$ is completely determined in terms of the $L$-parameter of $\pi_0$ by the work of Atobe and Gan in \cite{atobe2017local}, this theorem results in explicit determination of $m(\pi)$.

The main result of this paper (see Theorem \ref{thm:lifts}) provides an explicit description of $\theta(\pi,V_m)$ in terms of the Langlands classification. If $\pi$ is a quotient of the standard representation $\chi_V\delta_r\nu^{s_r} \times \dotsm \times \chi_V\delta_1\nu^{s_1} \rtimes \pi_0$, we write $\pi = L(\chi_V\delta_r\nu^{s_r}, \dotsc, \chi_V\delta_1\nu^{s_1}; \pi_0)$. We use the same notation even when the exponents $s_r,\dotsc,s_1$ are not sorted decreasingly; in that case, it is implied that the representations are to be sorted before taking the Langlands quotient; see Section \ref{subs:reps}. For the following theorem we also set $\theta_{l}(\pi) = \theta(\pi,V_{n+\epsilon-l})$. We fix an arbitrary generic character $\chi$ (see Section \ref{subs:generic}).
\begin{Thm}
Let $\pi = L(\chi_V\delta_r\nu^{s_r}, \dotsc, \chi_V\delta_1\nu^{s_1}; \pi_0)$ be an irreducible $\chi$-generic representation of $G(W_n)$. Let $l$ be an odd integer such that $\theta_{l}(\pi) \neq 0$. Then
\[
\chi_W\delta_r\nu^{s_r} \times \dotsm \times \chi_W\delta_1\nu^{s_1} \rtimes \theta_{l}(\pi_0) \twoheadrightarrow \theta_{l}(\pi).
\]
Furthermore, if $\theta_{l}(\pi_0) = L(\chi_W\delta_k'\nu^{s_k'}, \dotsc, \chi_W\delta_1'\nu^{s_1'}; \tau)$, then $\theta_{l}(\pi)$ is uniquely determined by
\[
\theta_{l}(\pi) = L(\chi_W\delta_r\nu^{s_r}, \dotsc, \chi_W\delta_1\nu^{s_1}, \chi_W\delta_k'\nu^{s_k'}, \dotsc, \chi_W\delta_1'\nu^{s_1'}; \tau).
\]
\end{Thm}
We now briefly describe the contents of this paper.

In Section \ref{sec:preli} we go over the basic notation and the results regarding the representation theory of the (quasi-split) classical $p$-adic groups. In Section \ref{sec:theta} we review the main results concerning theta correspondence in general. We also derive a number of useful corollaries (\ref{cor:theta_epi}-\ref{cor:shaving}) of Kudla's filtration (Theorem \ref{tm:Kudla}) which we use in subsequent sections. Section \ref{sec:first} contains the proof of Theorem \ref{thm:first_occ} which determines the first occurrence index. The proof relies heavily on Kudla's filtration and the standard module conjecture for classical groups (proven by Mui\'{c} in \cite{muic2001proof}) which asserts that any generic representation of a quasi-split classical group is in fact isomorphic to its standard module. In the fifth section we state our main result and prove it in some special cases. Section \ref{sec:interlude} contains a number of auxiliary technical results based on the work of Zelevinsky \cite{zelevinsky1980induced}. These results are used in Section \ref{sec:higher_lifts}, which contains the rest of the proof of Theorem \ref{thm:lifts}, providing a complete description of the lifts. Finally, in Section 8 we describe a method for constructing an interesting class of unitarizable representations, obtained by complementing Theorem \ref{thm:lifts} with the results of \cite{li1989singular} and \cite{lapid2004generic}.

The author would like to thank M.~Hanzer for the many useful discussions on the subject.

\section{Preliminaries}
\label{sec:preli}

\subsection{Groups}
\label{subs:groups}
Let $\F$ be a nonarchimedean local field of characteristic $0$ and let $|\cdot|$ be the absolute value on $\F$ (normalized as usual).

All the groups considered in this paper will be defined over $\F$. For $\epsilon = \pm 1$ fixed, we let
\[
\begin{cases}
W_n = \text{a }(-\epsilon)\text{-Hermitian space of dimension }n,\\
V_m = \text{an }\epsilon\text{-Hermitian space of dimension }m.
\end{cases}
\]
When $\epsilon = 1$, this means that $W_n$ is symplectic, whereas $V_m$ is a quadratic space. Denote by $G_n = G(W_n)$ and $H_m = H(V_m)$ the isometry groups of $W_n$ and $V_m$, respectively. Thus
\[
G(W_n) = 
\begin{cases}
\Sp(W_n) \text{ (the symplectic group)} \quad \text{if } \epsilon = 1,\\
\Ort(W_n) \text{ (the orthogonal group)} \quad \text{if } \epsilon = -1,
\end{cases}
\]
while the roles are reversed for $H(V_m)$. Furthermore, if $X$ is a vector space over $\F$, we denote by $\textnormal{GL} (X)$ the general linear group of $X$. Note that all the groups defined here are totally disconnected locally compact topological groups.

\subsection{Witt towers}
\label{subs:Witt}
Every Hermitian space $V_m$ has a Witt decomposition
\[
V_m = V_{m_0} + V_{r,r} \quad (m = m_0 + 2r),
\]
where $V_{m_0}$ is anisotropic and $V_{r,r}$ is split (i.e.~a sum of $r$ hyperbolic planes). The space $V_{m_0}$ is unique up to isomorphism, and so is the number $r \geqslant 0$, which is called the Witt index of $V_m$. The collection of spaces
\[
\mathcal{V} = \{V_{m_0} + V_{r,r} \colon r \geqslant 0\}
\]
is called a Witt tower. Since 
\[
\det(V_{m_0+2r}) = (-1)^r\det(V_{m_0}) \in \F^\times/(\F^\times)^2,
\]
the quadratic character
\[
\chi_V(x) = (x, (-1)^{\frac{m(m-1)}{2}}\det(V))_\F
\]
is the same for all the spaces $V$ in a single Witt tower (see \cite[\S V.1]{kudla1996notes}; here $(\cdot, \cdot)_\F$ denotes the Hilbert symbol).

\begin{Rem}
\label{rem:wittdim4}
In this paper we consider Witt towers of even dimension; this implies $m_0 = \dim(V_{m_0}) \in \{0,2,4\}$. However, if $\dim(V_{m_0}) = 4$, the orthogonal groups in the corresponding (so-called quaternionic) tower are not quasi-split, and thus have no generic representations (see Remark \ref{rem:quasi}). Consequently, when considering lifts of generic representations we only use this tower as a target for our theta lifts.
\end{Rem}

The symplectic spaces $W_n$ can be organized in a Witt tower in the same way. This case is somewhat simpler: since the only anisotropic symplectic space is the trivial one, there is only one tower of symplectic spaces. The corresponding character $\chi_W$ is trivial.

\subsection{Parabolic subgroups}
\label{subs:parabolic}
Let $V_m$ be a $\epsilon$-Hermitian space of dimension $m$ such that $m_0 \leq 2$ (cf.~Remark \ref{rem:wittdim4}). We may choose a subset $\{v_1, \dotsc, v_r, v_1',\dotsc, v_r'\}$ of $V_{m}$ such that $(v_i,v_j) = (v_i',v_j')=0$ and $(v_i,v_j') = \delta_{ij}$. Here $r$ denotes the Witt index of $V_m$ unless $V_{m}$ is a split quadratic space; in that case we set $2r= m-2$. We let $B=TU$ denote the standard $\F$-rational Borel subgroup of $H(V_m)$, i.e.~the subgroup of $H(V_m)$ stabilizing the flag
\[
0 \subset \text{span}\{v_1\} \subset \text{span}\{v_1,v_2\}\subset \dotsb \subset \text{span}\{v_1,v_2\dotsc,v_r\}.
\]
Furthermore, for any $t \leq r$ we set $U_t = \repspan\{v_1,\dotsc, v_t\}$ and $U_t' = \repspan\{v_1',\dotsc, v_t'\}$; we can then decompose
\[
V_m = U_t \oplus V_{m-2t} \oplus U_t'
\]
The subgroup $Q_t$ of $H(V_m)$ which stabilizes $U_t$ is a maximal parabolic subgroup of $H(V_m)$; it has a Levi decomposition $Q_t = M_tN_t$, where $M_t = \textnormal{GL}(U_t) \times H(V_{m-2t})$ is the Levi component, i.e. the subgroup of $Q_t$ which stabilizes $U_t'$ (here we often identify $\GL(U_t)$ with $ \GL_t(\F)$).

By letting $t$ vary, we obtain a set $\{Q_t: t \in\{1,\dotsc,r\}\}$ of standard maximal parabolic subgroups. By further partitioning $t$, we get the rest of the standard parabolic subgroups---generally, the Levi factor of a standard parabolic subgroup is of the form
\[
\GL_{t_1}(\F) \times \dotsm \times \GL_{t_k}(\F) \times H(V_{m-2t})  \quad (t = t_1 +\dotsb + t_k).
\]
We denote the maximal standard parabolic subgroups of $G(W_n)$ and $H(V_m)$ by $P_t$ and $Q_t$, respectively.

\subsection{Representations}
\label{subs:reps}
Let $G = G(W_n)$ be one of the groups described in \S\ref{subs:groups}. By a representation of $G$ we mean a pair $(\pi,V)$ where $V$ is a complex vector space and $\pi$ is a homomorphism $G \to \GL(V)$. With $V_{\infty}$ we denote the subspace of $V$ comprised of all the smooth vectors, i.e.~those having an open stabilizer in $G$.
If $V = V_{\infty}$, we say that the representation $(\pi,V)$ is smooth. Unless otherwise stated, we will assume that all the representations are smooth; the category of all smooth complex representations of $G$ will be denoted by $\mathcal{A}(G)$. The set of equivalence classes of irreducible representations of $G$ will be denoted by $\Irr(G)$.

For each parabolic subgroup $P=MN$ of $G$ we have the (normalized) induction and localization (Jacquet) functors, $\Ind_P^G\colon \mathcal{A}(M) \to \mathcal{A}(G)$ and $R_P \colon \mathcal{A}(G) \to \mathcal{A}(M)$. These are connected by the standard Frobenius reciprocity
\[
\Hom_G(\pi, \Ind_P^G(\pi')) \cong \Hom_M(R_P(\pi), \pi')
\]
and by the second (Bernstein) form of Frobenius reciprocity,
\[
\Hom_G(\Ind_P^G(\pi'), \pi) \cong \Hom_M(\pi', R_{\overline{P}}(\pi))
\]
(here $\overline{P} = M\overline{N}$ is the parabolic subgroup opposite to $P$).

If $P=MN$ is a parabolic subgroup of $G(W_n)$ with Levi factor $M = \GL_{t_1}(\F) \times \dotsm \times \GL_{t_k}(\F) \times G(W_{n-2t})$, we write
\[
\tau_1 \times \dotsm \times \tau_k \rtimes \pi_0
\]
for $\Ind_P^G(\tau_1 \otimes \dotsm \otimes \tau_k \otimes \pi_0)$, where $\tau_i$ is a representation of $\GL_{t_i}(\F)$ and $\pi_0$ is a representation of $G(W_{n-2t})$ (with $t = t_1 +\dotsb+t_k$).

To obtain a complete list of irreducible representations of $G(W_n)$, we use the Langlands classification: let $\delta_i \in \GL_{t_i}(\F), i = 1,\dotsc, r$ be irreducible discrete series representations, and let $\tau$ be an irreducible tempered representation of $G(W_{n-2t})$ (for $t=t_1+\dotsb+t_r$). Any representation of the form
\[
\nu^{s_r}\delta_r \times \dotsb \times \nu^{s_1}\delta_1 \rtimes \tau,
\]
where $s_r \geqslant \dotsb \geqslant s_1 > 0$ (and where $\nu$ denotes the character $\lvert\det\rvert$ of the corresponding general linear group) is called a standard representation (or a standard module). It possesses a unique irreducible quotient, the so-called Langlands quotient, denoted by $L(\nu^{s_r}\delta_r \times \dotsb \times \nu^{s_1}\delta_1 \rtimes \tau)$. Occasionally, we will also write $L(\nu^{s_r}\delta_r, \dotsc, \nu^{s_1}\delta_1; \tau)$, implying that the representations $\{\nu^{s_r}\delta_r, \dotsc, \nu^{s_1}\delta_1\}$ are to be sorted decreasingly with respect to $s_i$'s before taking the quotient. Conversely, every irreducible representation can be represented as the Langlands quotient of a unique standard representation. In this way, we obtain a complete description of $\Irr(G(W_n))$.

We will use this (quotient) form of the Langlands classification interchangeably with the subrepresentation form, by means of the following lemma \cite[Lemma 2.2]{atobe2017local}.

\begin{Lem}
\label{lemma:MVWinv}
Let $P$ be a standard parabolic subgroup of $G(W_n)$ with Levi component equal to $\GL_{t_1}(\F) \times \dotsm \times \GL_{t_r}(\F) \rtimes G(W_{n_0})$. Then, for $\tau_i \in \Irr(\GL_{t_i}(\F))$, $\pi_0 \in \Irr(G(W_{n_0}))$ and $\pi \in \Irr(G(W_{n}))$ the following statements are equivalent:
\begin{enumerate}[(i)]
\item $\pi \hookrightarrow \tau_1 \times \dotsm \times \tau_r \rtimes \pi_0$; 
\item $\tau_1^\vee \times \dotsm \times \tau_r^\vee \rtimes \pi_0 \twoheadrightarrow \pi$.
\end{enumerate}
\end{Lem}
Here $\tau^\vee$ denotes the contragredient representation. When dealing with tempered representations, we will often need the following two lemmas. The first one is due to Harish-Chandra (see Lemma 2.1 \cite{muic2008theta}):
\begin{Lem}
\label{lemma:temp_supp}
Let $\pi_0 \in \Irr(G(W_n))$ be a tempered representation. Then there exist irreducible discrete series representations $\delta_1, \dotsc, \delta_k, \pi_{00}$ such that $\pi_{0} \hookrightarrow \delta_1 \times \dotsm \times \delta_k \rtimes \pi_{00}$. If $\delta'_1, \dotsc, \delta'_{k'}, \pi'_{00}$ is another sequence of discrete series representations such that $\pi_{0} \hookrightarrow \delta'_1 \times \dotsm \times \delta'_{k'} \rtimes \pi'_{00}$, then $k'=k$, $\pi'_{00} = \pi_{00}$ and the sequence $\delta'_1, \dotsc, \delta'_k$ is obtained from $\delta_1, \dotsc, \delta_k$ by permuting the terms and replacing some of them with their contragredients. The multiset $\{\delta_1, \dotsc, \delta_k, \delta_1^\vee, \dotsc, \delta_k^\vee, \pi_{00}\}$ is called the tempered support of $\pi_0$.
\end{Lem}
The second lemma we need is a result of Goldberg (\cite{goldberg1994reducibility}, Theorems 6.4 and 6.5)
\begin{Lem}
\label{lemma:goldberg}
Let $\delta_1, \dotsc, \delta_k, \pi_{00}$ be a sequence of discrete series representations. Then the induced representation  $\delta_1 \times \dotsm \times \delta_k \rtimes \pi_0$ is a direct sum of mutually non-isomorphic tempered representations. It is of length $2^L$ where $L$ is the number of non-equivalent $\delta_i$ such that $\delta_i \rtimes \pi_{00}$ reduces.
\end{Lem}
Both results are originally stated for connected groups, but can easily be extended to the non-connected case of $\Ort(V)$ (see e.g.~\cite{muic2008theta}, Lemma 2.1 and 2.3).

\subsection{Computing Jacquet modules}
\label{subs:computingJM}
We need to compute the Jacquet modules of various representations on a number of occasions. We let $R_n, n\geq 0$ denote the Grothendieck group of admissible representations of $\GL_n(\F)$ of finite length; we also set $
R = \oplus_{n\geq 0} R_n$.

For $\pi_1 \in \Irr(\GL_{n_1}(\F))$ and $\pi_2 \in \Irr(\GL_{n_2}(\F))$ the pairing
\[
(\pi_1, \pi_2) \mapsto \pi_1 \times \pi_2
\]
defines an additive mapping $\times \colon R_{n_1} \times R_{n_2} \to R_{n_1+n_2}$. We extend the mapping $\times$ to a multiplication on $R$ in a natural way.

On the other hand, for any $\pi \in \Irr(\GL_n(\F))$ we may identify $R_{P_k}(\pi)$ with its semi-simplification in $R_k \otimes R_{n-k} \hookrightarrow R\otimes R$ (here $P_k$ temporarily denotes the $k$-th maximal standard parabolic subgroup of $\GL_n(\F)$). We define
\[
m^*(\pi) = (\mathbb{1}\otimes \pi) \oplus (\pi \otimes \mathbb{1})  \oplus  \sum_{k=1}^{n-1}R_{P_{k}}(\pi) \in R\otimes R
\]
and extend $m^*$ to an additive map $R \to R\otimes R$. The basic fact due to Zelevinsky (see Section 1.7 of \cite{zelevinsky1980induced} for additional details) is that
\begin{equation}
\label{eq_Hopf}
m^*(\pi_1 \times \pi_2) = m^*(\pi_1) \times m^*(\pi_2).
\end{equation}
In most cases, we will consider $m^*(\pi)$ when $\pi = \langle \nu^b\rho, \nu^a\rho\rangle$, i.e.~the discrete series representation attached to the segment $[\nu^a\rho, \nu^b\rho]$, or $\pi = \langle \nu^a\rho, \nu^b\rho\rangle$, i.e.~the Langlands quotient of $\nu^b\rho \times \dotsm \times \nu^a\rho$, where $b-a \in \mathbb{Z}_{\geq 0}$ (see Section \ref{sec:interlude} for notation). In those cases, we have
\begin{align}
m^*(\langle \nu^b\rho, \nu^a\rho\rangle) &= \sum_{i=a-1}^b \langle \nu^b\rho, \nu^{i+1}\rho\rangle \otimes \langle \nu^i\rho, \nu^a\rho\rangle,\label{eq_JMdelta}\\
m^*(\langle \nu^a\rho, \nu^b\rho\rangle) &= \sum_{i=a-1}^b \langle \nu^a\rho, \nu^i\rho\rangle \otimes \langle \nu^{i+1}\rho, \nu^b\rho\rangle.\label{eq_JMzeta}
\end{align}
It is important to note that in \eqref{eq_JMdelta} (resp.~\eqref{eq_JMzeta}) we set $\langle \nu^{c-1}\rho, \nu^c\rho\rangle = \mathbb{1} \in \Irr(\GL_0(\F))$ (resp.$\langle \nu^{c}\rho, \nu^{c-1}\rho\rangle = \mathbb{1}$) for any $c\in \mathbb{R}$.

This theory was extended by Tadi\'{c} to the case of classical groups in \cite{tadic1995structure}. For any $\pi \in \Irr(G_n)$ we let $\mu^*(\pi)$ be the sum of the semi-simplifications of $R_P(\pi)$ when $P$ varies over the set of maximal standard parabolic subgroups of $G_n$ as well as $P=G_n$. The relevant formula is now
\begin{equation}
\label{eq_Tadic}
\mu^*(\delta \rtimes \pi) = M^*(\delta) \rtimes \mu^*(\pi).
\end{equation}
The definition of $M^*$ can be found in \cite[Theorem 5.4]{tadic1995structure}, but we shall need it here only in the special case when $\delta$ is essentially square integrable representation $\langle \nu^{k}\rho, \rho\rangle$ for $k\geq 0$; in this case, we have (\cite[\S 14]{tadic2012reducibility})
\begin{equation}
\label{eq_M}
M^*(\langle \nu^{k}\rho, \rho\rangle)= \sum_{i=-1}^k\sum_{j=i}^k \langle\rho^\vee, \nu^{-i}\rho^\vee\rangle \times \langle \nu^{k}\rho, \nu^{j+1}\rho\rangle  \otimes \langle \nu^{j}\rho, \nu^{i+1}\rho\rangle .
\end{equation}

\subsection{Local Langlands Correspondence}
\label{subs:LLC}
Another way of classifying the irreducible representations of $G(W_n)$ is by means of the Local Langlands Correspondence (LLC). We use it mainly to harvest the results on lifts of tempered representations established recently by Atobe and Gan in \cite{atobe2017local}. Without going into detail, we give a brief description of the basic features of LLC; a concise overview of the theory along with the key references can be found in appendices A and B of \cite{atobe2017local}.

The LLC parametrizes $\Irr(G(W_n))$ by representations of the Weil-Deligne group, $\mathrm{WD}_\F = \mathrm{W}_\F \times \SL_2(\C)$ (here $W_\F$ denotes the Weil group of $\F$). More precisely, we define $\Phi(G(W_n))$, for any even $n$, as a set of equivalence classes:
\[
\begin{cases}
\Phi(\Sp(W_n)) &= \{\phi \colon \mathrm{WD}_\F \to \SO(n+1,\C)\} / \cong,\\
\Phi(\Ort(W_n)) &= \{\phi \colon \mathrm{WD}_\F \to \Ort(n,\C) | \det(\phi) = \chi_W\} / \cong.
\end{cases}
\]
The irreducible representations of $G(W_n)$ are then parametrized by the so-called $L$-pa\-ra\-me\-ters, i.e. pairs of the form $(\phi,\eta)$, where $\phi \in \Phi(G(W_n))$, and $\eta$ is a character of the (finite) component group of the centralizer of $\text{Im}(\phi)$. The set of representations which correspond to the same $\phi$ is called the $L$-packet attached to $\phi$. 

Any $\phi \in \Phi(G(W_n))$ can be decomposed as 
\[
\phi = \bigoplus_{n \geqslant 1} \phi_n \otimes S_n, 
\]
where $\phi_n$ is a representation of $W_\F$, whereas $S_n$ denotes the unique algebraic representation of $\SL_2(\C)$ of dimension $n$. Tempered representations are parametrized by pairs $(\phi,\eta)$ in which $\phi(W_\F)$ is bounded; among those, the multiplicity free parameters correspond to discrete series representations.

Note that, unlike $\phi$, the choice of $\eta$ is non-canonical: it depends on the choice of a Whittaker datum of $G(W_n)$. This choice will be fixed like in Remark B.2 of \cite{atobe2017local}.

\subsection{Generic representations}
\label{subs:generic}
To define generic representations, we assume that $H(V_m)$ is quasi-split, i.e.~that it has a Borel subgroup defined over $\F$.
\begin{Rem}
\label{rem:quasi}
The isometry groups introduced in Section \ref{subs:groups} are all quasi-split, with the exception of $\Ort(V_m)$ when the anisotropic subspace $V_{m_0}$ is of dimension $4$.
\end{Rem}
We retain the notation of Section \ref{subs:parabolic}; in particular, we let $B = TU$ be the standard Borel subgroup of $H$ as fixed in that section. We now fix a non-trivial additive character $\psi$ of $\F$. If $H(V_m) = \Sp(V_m)$ we may now define a generic character $\chi\colon U \to \F$ by choosing an element $c\in \F^\times$ and setting
\[
\chi_c(u) = \psi(\langle uv_2,v_1' \rangle_V + \dotsb + \langle uv_m,v_{m-1}' \rangle_V + c\langle uv_m',v_m' \rangle_V)
\]
where $\langle \cdot,\cdot \rangle_V$ denotes the bilinear form on $V_m$. The map $c\mapsto \chi_c$ establishes a bijection
\[
\F^\times/{\F^\times}^2 \to \{T\text{-orbits of generic characters of }U\}.
\]
We point out that this bijection depends on our choice of $\psi$ in the symplectic case. We say that a representation $(\pi,V)$ of $\Sp(V_m)$ is $\chi$-generic if there is a non-trivial linear functional $l_\pi \colon V \to \C$ such that
\[
l_\pi(\pi(u)v) = \chi(u)l_\pi(v).
\]
for all $v \in V$ and $u\in U$. When $H(V_m)=\Ort(V_m)$, we define $\chi$-generic representations in a similar manner; see \cite{Atobe2018} for additional details. From now on, we fix an arbitrary generic character $\chi$. 

The following theorem contains the most important properties of $\chi$-generic representations which we often use. The first two (established by F.~Rodier \cite{rodier1973whittaker}) are known as the heredity and the uniqueness of the Whittaker model, respectively. The third one is the standard module conjecture, established by G.~Mui\'{c} in \cite{muic2001proof}.
\begin{Thm}
\label{thm:gen_properties}
\begin{enumerate}[(i)]
\item If $\tau_i \in \Irr(\GL_{t_i}(\F)), i=1,\dotsc,r$ are irreducible generic representations, and $\pi_0$ is an irreducible representation of $H(V_m)$, then 	$\tau_1 \times \dotsm \times \tau_k \rtimes \pi_0$ is $\chi$-generic if and only if $\tau_1 \otimes \dotsm \otimes \tau_k \otimes \pi_0$ is $\chi$-generic.
\item If $H(V_m) = \Sp(V_m)$, and $\pi_0, \tau_i \in \Irr(\GL_{t_i}(\F)), i=1,\dotsc,r$ are irreducible $\chi$-generic representations, then $\tau_1 \times \dotsm \times \tau_k \rtimes \pi_0$ contains a unique irreducible $\chi$-generic subquotient, which has multiplicity one.
\item The standard module of any irreducible generic representation of $H(V_m)$ is itself irreducible.
\end{enumerate}
\end{Thm}
All of the above properties were originally proven in the setting of classical groups in \cite{muic2001proof}, and later extended to connected quasi-split groups in \cite{heiermann2007standard}, \cite{heiermann2013tempered}. Whereas (i) and (iii) can easily be extended to include the $\Ort(V)$ case (see \cite[Theorem 6.4]{hanzer2009generalized}), property (ii) fails in the non-connected case, so we need to be careful when dealing with representations of $\Ort(V)$. The third statement can be viewed as a consequence of the so-called generalized injectivity conjecture, established by M.~Hanzer in \cite{hanzer2009generalized}.

We often combine (iii) with the following result \cite[Introduction]{muic2005reducibility}:
\begin{Prop}
\label{prop:std_mod_reduc}
A standard representation of the form $\nu^{s_r}\delta_r \times \dotsb \times \nu^{s_1}\delta_1 \rtimes \tau$ reduces if and only if one of the following holds
\begin{enumerate}[(i)]
\item $\nu^{s_i}\delta_i \times \nu^{s_j}\delta_j$ reduces for some pair $i\neq j$;
\item $\nu^{s_i}\delta_i \times \nu^{-s_j}{\delta_j}^\vee$ reduces for some pair $i\neq j$;
\item $\nu^{s_i}\rtimes \tau$ reduces for some $i$.
\end{enumerate}
\end{Prop}

\section{Theta correspondence}
\label{sec:theta}
In this section, we review the basic facts concerning the local theta correspondence established in \cite{kudla1986local}, \cite{howe1979beyond} and \cite{waldspurger1990}. We also fix the notation, roughly following \cite{kudla1996notes}.
\subsection{Howe duality}
\label{subs:Howe}
Let $\omega_{m,n}$ be the Weil representation of $G(W_n) \times H(V_m)$. The Weil representation depends on the choice of a non-trivial additive character $\psi\colon\F\to\C$. As this character is fixed throughout, we omit it from the notation. Similarly, if the dimensions $m$ and $n$ are known, we will often simply write $\omega$ instead of $\omega_{m,n}$.

For any $\pi \in \Irr(G(W_n))$, a basic structural fact about the Weil representation (\cite[Chapter II, III.4]{MVW}) guarantees that the maximal $\pi$-isotypic quotient of $\omega_{m,n}$ is of the form
\[
\pi \otimes \Theta(\pi,V_m)
\]
for a certain smooth representation $\Theta(\pi,V_m)$ of $H(V_m)$, called the full theta lift of $\pi$. When the target Witt tower is fixed, we will often denote it by $\Theta(\pi,m)$ or, more often, by $\Theta_l(\pi)$, where $l = n + \epsilon - m$.

The key result which establishes the theta correspondence is the following:
\begin{Thm}[Howe duality]
\label{thm:Howeduality}
If $\Theta(\pi,V_m)$ is non-zero, it possesses a unique irreducible quotient, denoted by $\theta(\pi,V_m)$.
\end{Thm}
Originally conjectured by Howe in \cite[p.~279]{howe1979beyond}, it was first proven by Waldspurger \cite{waldspurger1990} when the residual characteristic of $\F$ is different from $2$, and by Gan and Takeda \cite{gan2016proof} in general.
The representation $\theta(\pi,V_m)$ is called the (small) theta lift of $\pi$; like the full lift, it will also be denoted by $\theta(\pi,m)$ and $\theta_l(\pi)$.

For future reference, we state the following simple but useful fact (\cite[Lemma 1.1]{muic2008structure}):
\begin{Lem}
\label{lemma:ThetaisHom}
For $\pi \in \Irr(G(W_n))$ we have
\[
\Theta(\pi,m)^\vee = \Hom_{G_n}(\omega_{m,n},\pi)_\infty.
\]
\end{Lem}

\subsection{First occurrence in towers}
\label{subs:first_occurrence}
The study of theta correspondence in towers is motivated by the following facts (\cite[Propositions 4.1 and 4.3]{kudla1996notes}):
\begin{Prop}
\label{prop:towers}
Let $\pi$ be an irreducible representation of $G(W_n)$. Fix a Witt tower $\mathcal{V} = (V_m)$.
\begin{enumerate}[(i)]
\item If $\Theta(\pi,V_m) \neq 0$, then $\Theta(\pi,V_{m+2r}) \neq 0$ for all $r\geqslant 0$.
\item For $m$ large enough, we have $\Theta(\pi,V_m) \neq 0$.
\end{enumerate}
\end{Prop}
The above proposition implies that we can define, for any Witt tower $\mathcal{V} = (V_m)$,
\[
m_\mathcal{V}(\pi) = \min\{m \geqslant 0: \Theta(\pi,V_m) \neq 0\}.
\]
This number (also denoted by $m(\pi)$ when the choice of $\mathcal{V}$ is implicit) is called the first occurrence index of $\pi$. Note that we are using the term ``index"	 here to signify the dimension, although it would be more appropriate to use it for the Witt index of the corresponding space.

An important result which helps us compute the first occurrence indices is the so-called conservation relation. The Witt towers of quadratic spaces can be appropriately organized into pairs, with the towers comprising a pair denoted by $\mathcal{V}^+$ and $\mathcal{V}^-$ (a complete list of pairs of dual towers can be found in \cite[Chapter V]{kudla1996notes}). Thus, instead of observing just one target tower, we can simultaneously look at two of them. This way, for each $\pi \in \Irr(G(W_n))$ we get two corresponding first occurrence indices, $m^+(\pi)$ and $m^-(\pi)$.

If $\epsilon = -1$ so that $W_n$ is a quadratic space, we proceed as follows: since $G(W_n)$ is now equal to $\Ort(W_n)$, any $\pi \in \Irr(G(W_n))$ is naturally paired with its twist, $\det \otimes \pi$. This allows us to define
\begin{align*}
m^+(\pi) &= \min\{m(\pi), m(\det\otimes\pi)\},\\
m^-(\pi) &= \max\{m(\pi), m(\det\otimes\pi)\}.
\end{align*}
We are now able to set
\[
m^{\text{down}}(\pi) =  \min\{m^+(\pi), m^-(\pi)\}, \quad m^{\text{up}}(\pi) =  \max\{m^+(\pi), m^-(\pi)\}
\]
regardless of whether $\epsilon = 1$ or $\epsilon = -1$. Note that when $W_n$ is a quadratic space, we have $m^{\text{down}}(\pi) = m^+(\pi) $ and $ m^{\text{up}}(\pi) = m^-(\pi)$. The conservation relation (first conjectured by Kudla and Rallis in \cite{kudla2005first}, completely proven by Sun and Zhu in \cite{sun2015conservation}) states that
\[
m^{\text{up}}(\pi) + m^{\text{down}}(\pi) = 2n + 2\epsilon + 2.
\]
The tower in which $m(\pi) = m^{\text{down}}(\pi)$ (resp.~$m^{\text{up}}(\pi)$) will often be called the going-down (resp.~going-up) tower.
\subsection{Kudla's filtration}
\label{subs:Kudla}
One of our main tools is Kudla's filtration of $R_P(\omega)$, the Jacquet module of the Weil representation (\cite[Theorem 2.8]{kudla1986local}). We state it here (formulated as in \cite[Theorem 5.1]{atobe2017local}) along with a few useful corollaries.
\begin{Thm}
\label{tm:Kudla}
The Jacquet module $R_{P_k}(\omega_{m,n})$ possesses a $\GL_k(\F)\times G(W_{n-2k}) \times H(V_m)$-equivariant filtration
\[
R_{P_k}(\omega_{m, n}) = R^0 \supset R^1 \supset \dotsb \supset R^k \supset R^{k+1} = 0
\]
in which the successive quotients $J^a = R^a/R^{a+1}$ are given by
\[
J^a = \Ind_{P_{k-a,a}\times G_{n-2k}\times Q_a}^{\GL_k(\F)\times G_{n-2k}\times H_m}\left({\chi_V|\emph{\text{det}}_{\GL_{k-a}(\F)}|^{\lambda_{k-a}}\otimes \Sigma_a \otimes \omega_{m-2a, n-2k}}\right),
\]
where
\begin{itemize}
\item $\lambda_{k-a} = (m-n+k-a-\epsilon)/2$;
\item $P_{k-a,a}$ is the standard parabolic subgroup of $\GL_k(\F)$ with Levi component $\GL_{k-a}(\F) \times \GL_a(\F)$;
\item $\Sigma_a = C_c^\infty(\GL_a(\F))$, the space of locally constant compactly supported functions on $\GL_a(\F)$. The action of $\GL_a(\F) \times \GL_a(\F)$ on $\Sigma_a$ is given by
\[
[(g,h)\cdot f](x) = \chi_V(\det(g))\chi_W(\det(h))f(g^{-1}\cdot x\cdot h).
\]
\end{itemize}
Here the first $\GL_a(\F)$ factor comes from the Levi factor of $P_{k-a,a}$, whereas the second comes from the Levi factor $Q_a$. If $m-2a$ is less than the dimension of the first (anisotropic) space in $\mathcal{V}$, we put $R^a=J^a=0$.
\end{Thm}

We will often use the following proposition (see Corollary 3.2 of \cite{muic2008structure} and also Proposition 5.2 of \cite{atobe2017local}) derived from the previous theorem:
\begin{Prop}
\label{prop:Muic_isotype}
Let $\delta$ be an irreducible essentially square integrable representation of $\GL_k(\F)$ and $\pi_0 \in \Irr (G_{n-2k})$, for some $k>0$. Then the space $\Hom_{\GL_k(\F)\times G_{n-2k}}(J^a, \chi_V\delta^\vee\otimes \pi_0)_\infty$, viewed as a representation of $H_m$, is isomorphic to
\[
\begin{cases}
\chi_W^{-1}\delta^\vee \rtimes \Hom_{G_{n-2k}}(\omega_{m-2k,n-2k}, \pi_0)_\infty, &\text{if } a = k, \\
\chi_W^{-1}\textnormal{St}_{k-1}\nu^{\frac{k-l+1}{2}} \rtimes \Hom_{G_{n-2k}}(\omega_{m-2k+2,n-2k}, \pi_0)_\infty,&\text{if } a = k-1 \text{ and } \delta = \textnormal{St}_k\nu^{\frac{l-k}{2}}, \\
0, &\text{otherwise.}
\end{cases}
\]
\end{Prop}
Recall that, in the above proposition, we have $\Hom_G(\omega,\pi)_\infty = \Theta(\pi)^\vee$. Furthermore, $\textnormal{St}_k$ denotes the so-called Steinberg representation of $\GL_k(\F)$, the square integrable representation attached to the segment $[|\cdot|^{\frac{1-k}{2}},|\cdot|^{\frac{k-1}{2}}]$. Thus $\textnormal{St}_k\nu^{\frac{l-k}{2}} = \langle|\cdot|^\frac{l-1}{2}, |\cdot|^{\frac{l+1}{2}-k}\rangle$ (see the beginning of Section \ref{sec:interlude} for the notation used here). We point out that the condition $l>0$ given in \cite[Proposition 5.2]{atobe2017local} is not necessary (cf.~\cite[Corollary 3.2]{muic2008structure}).

We now list a few useful corollaries of Proposition \ref{prop:Muic_isotype}. The first one is \cite[Corollary 5.3]{atobe2017local} (see also \cite[Corollary 3.2]{muic2008structure}).
\begin{Cor}
\label{cor:theta_epi}
Let $\pi\in \Irr(G_n), \pi_0 \in \Irr(G_{n-2k})$ and let $\delta$ be an irreducible essentially square integrable representation of $\GL_k(\F)$. Assume that $\delta \ncong \textnormal{St}_k\nu^{\frac{l-k}{2}}$, where $l = n - m + \epsilon$. Then
\[
\chi_V\delta \rtimes \pi_0 \twoheadrightarrow \pi
\]
implies
\[
\chi_W\delta \rtimes \Theta_l(\pi_0) \twoheadrightarrow \Theta_l(\pi).
\]
\end{Cor}
The second corollary we state is a slight modification of the first: this time, we are unable to obtain information about the full lift $\Theta_l(\pi)$, but we allow the special case $\delta \cong \textnormal{St}_k\nu^{\frac{l-k}{2}}$:
\begin{Cor}
\label{cor:shaving}
Let $\delta$ be an irreducible essentially square integrable representation of $\GL_k(\F)$ and let $\pi\in \Irr(G_n), \pi_0 \in \Irr(G_{n-2k})$ be such that
\[
\chi_V\delta \rtimes \pi_0 \twoheadrightarrow \pi.
\]
Furthermore, let $A$ be an admissible representation of a general linear group. Assume that an irreducible representation $\sigma$ satisfies
\[
A \rtimes \Theta_l(\pi) \twoheadrightarrow \sigma,
\]
where $l = n - m + \epsilon$. If $\delta \ncong \textnormal{St}_k\nu^{\frac{l-k}{2}}$, then
\[
\label{eq_shaving1}\tag{i}
A \times \chi_W\delta \rtimes \Theta_l(\pi_0) \twoheadrightarrow \sigma.
\]
If $\delta \cong \textnormal{St}_k\nu^{\frac{l-k}{2}} \cong \langle|\cdot|^{\frac{l-1}{2}}, |\cdot|^{\frac{l+1}{2}-k}\rangle$, then either (i) is true, or the following holds:
\[
\label{eq_shaving2}\tag{ii}
A \times \chi_W\langle|\cdot|^{\frac{l-3}{2}}, |\cdot|^{\frac{l+1}{2}-k}\rangle \rtimes \Theta_{l-2}(\pi_0) \twoheadrightarrow \sigma.
\]
\end{Cor}

\begin{proof}
By assumption, we have $\pi \hookrightarrow \chi_V\delta^\vee \rtimes \pi_0$, so
\begin{align*}
\Theta_l(\pi)^\vee &\cong \Hom_{G_n}(\omega_{m,n},\pi)_{\infty}\\
&\hookrightarrow  \Hom_{G_n}(\omega_{m,n},\chi_V\delta^\vee \rtimes \pi_0)_{\infty}\\
&\cong \Hom_{\GL_k(\F) \times G_{n-2k}}(R_{P_k}(\omega_{m,n}),\chi_V\delta^\vee \otimes \pi_0)_{\infty}.
\end{align*}
We now use Kudla's filtration to analyze $R_{P_k}(\omega_{m,n})$. For each index $a = 0, \dotsc, k$ we have an exact sequence
\[
0 \to R^{a+1 }\to R^{a} \to J^{a} \to 0.
\]
Applying the $\Hom_{\GL_k(\F) \times G_{n-2k}}(\ \cdot\ , \delta^\vee \otimes \pi_0)_\infty$ functor we get
\begin{equation}
\begin{aligned}
\label{eq_exact}
0 \to \Hom_{\GL_k(\F) \times G_{n-2k}}(J^{a},\chi_V\delta^\vee \otimes \pi_0)_{\infty} &\to \Hom_{\GL_k(\F) \times G_{n-2k}}(R^{a}, \chi_V\delta^\vee \otimes \pi_0)_{\infty}\\
&\to \Hom_{\GL_k(\F) \times G_{n-2k}}(R^{a+1}, \chi_V\delta^\vee \otimes \pi_0)_{\infty}.
\end{aligned}
\end{equation}
Since we know, by Proposition \ref{prop:Muic_isotype}, that the space $\Hom_{\GL_k(\F) \times G_{n-2k}}(J^{a}, \chi_V\delta^\vee \otimes \pi_0)_{\infty}$ is trivial for $a = 0, \dotsc, k-2$, this leads to an inclusion
\[
\Hom_{\GL_k(\F) \times G_{n-2k}}(R_{P_k}(\omega_{m,n}),\chi_V\delta^\vee \otimes \pi_0)_{\infty} \hookrightarrow \Hom_{\GL_k(\F) \times G_{n-2k}}(R^{k-1},\chi_V\delta^\vee \otimes \pi_0)_{\infty}.
\]
In particular, we have $\Theta_l(\pi)^\vee \hookrightarrow \Hom_{\GL_k(\F) \times G_{n-2k}}(R^{k-1}, \chi_V\delta^\vee \otimes \pi_0)_{\infty}$. Inducing with $A^\vee$, we get 
\[
A^\vee \rtimes \Theta_l(\pi)^\vee \hookrightarrow A^\vee \rtimes \Hom_{\GL_k(\F) \times G_{n-2k}}(R^{k-1}, \chi_V\delta^\vee \otimes \pi_0)_{\infty}.
\]
By assumption, we have $\sigma^\vee {\hookrightarrow} A^\vee \rtimes \Theta_l(\pi)^\vee$, so there is an injective equivariant map
\[
f\colon \sigma^\vee {\hookrightarrow} A^\vee\rtimes\Hom_{\GL_k(\F) \times G_{n-2k}}(R^{k-1}, \chi_V\delta^\vee \otimes \pi_0)_{\infty}.
\]
On the other hand, we may set $a=k-1$ in \eqref{eq_exact} and induce to get
\begin{align*}
0 &\to A^\vee\rtimes\Hom_{\GL_k(\F) \times G_{n-2k}}(J^{k-1}, \chi_V\delta^\vee \otimes \pi_0)_{\infty}\\
&\stackrel{g}{\to} A^\vee\rtimes\Hom_{\GL_k(\F) \times G_{n-2k}}(R^{k-1}, \chi_V\delta^\vee \otimes \pi_0)_{\infty} \stackrel{h}{\to} A^\vee\rtimes\Hom_{\GL_k(\F) \times G_{n-2k}}(J^{k}, \chi_V\delta^\vee \otimes \pi_0)_{\infty}.
\end{align*}
We now consider two options:
\begin{enumerate}[(i)]
\item If $\text{Im}(f) \cap \text{Ker}(h) = 0$, then we have an injective map
\[
h\circ f\colon \sigma^\vee \hookrightarrow A^\vee\rtimes\Hom_{\GL_k(\F) \times G_{n-2k}}(J^{k}, \chi_V\delta^\vee \otimes \pi_0)_{\infty}.
\]
Proposition \ref{prop:Muic_isotype} describes $\Hom_{\GL_k(\F) \times G_{n-2k}}(J^{k}, \chi_V\delta^\vee \otimes \pi_0)_{\infty}$; by taking the contragredient we get
\[
A \rtimes \chi_W\delta \rtimes \Theta_l(\pi_0) \twoheadrightarrow \sigma.
\]
Note that $\delta \ncong  \textnormal{St}_k\nu^{\frac{l-k}{2}}$ implies $\Hom_{\GL_k(\F) \times G_{n-2k}}(J^{k-1}, \chi_V\delta^\vee \otimes \pi_0)_{\infty}=0$. In that case, $ \text{Ker}(h) = \text{Im}(g) = 0$, so we always have the above result. If $\delta \cong  \textnormal{St}_k\nu^{\frac{l-k}{2}}$, we may have $\text{Im}(f) \cap \text{Ker}(h) \neq 0$.

\item If $\text{Im}(f) \cap \text{Ker}(h) \neq 0$, then the irreducibility of $\sigma$ implies $\sigma^\vee \hookrightarrow  \text{Ker}(h)$. By the exactness of the above sequence we have $\text{Ker}(h) = \text{Im}(g)$, and since $g$ is injective, we also have $ \text{Im}(g) \cong A^\vee \rtimes \Hom_{\GL_k(\F) \times G_{n-2k}}(J^{k-1}, \chi_V\delta^\vee \otimes \pi_0)_{\infty}$. Thus, we can write
\[
\sigma^\vee \hookrightarrow A^\vee \rtimes \Hom_{\GL_k(\F) \times G_{n-2k}}(J^{k-1}, \chi_V\delta^\vee \otimes \pi_0)_{\infty}
\]
from which, by looking at the contragradient (and using Proposition \ref{prop:Muic_isotype}), we arrive at
\[
A \rtimes \chi_W\langle|\cdot|^{\frac{l-3}{2}}, |\cdot|^{\frac{l+1}{2}-k}\rangle \rtimes \Theta_{l-2}(\pi_0) \twoheadrightarrow \sigma.
\]
\end{enumerate}
\end{proof}
We point out a special case of the above corollary, when $A$ is the trivial representation of the trivial group and $\sigma = \theta_l(\pi)$. Corollary \ref{cor:shaving} then shows we have
\[
\label{eq_shaving01}\tag{i}
\chi_W\delta \rtimes \Theta_{l}(\pi_0) \twoheadrightarrow \theta_l(\pi)
\]
unless $\delta \cong \textnormal{St}_k\nu^{\frac{l-k}{2}} \cong \langle|\cdot|^{\frac{l-1}{2}}, |\cdot|^{\frac{l+1}{2}-k}\rangle$. In that case, either (i) holds, or
we have
\[
\label{eq_shaving02}\tag{ii}
\chi_W\langle|\cdot|^{\frac{l-3}{2}}, |\cdot|^{\frac{l+1}{2}-k}\rangle \rtimes \Theta_{l-2}(\pi_0) \twoheadrightarrow \theta_l(\pi).
\]

\begin{Rem}
\label{rem:parentnotation}
At some point it will be useful to use the same notation for the outcomes of both options (i) and (ii) in Corollary \ref{cor:shaving}. With this in mind, we set
\[
(\delta) = 
\begin{cases}
\delta, &\text{if we used option (i)}\\
\langle|\cdot|^{\frac{l-3}{2}}, |\cdot|^{\frac{l+1}{2}-k}\rangle, &\text{if we used option (ii).}
\end{cases}
\]
\end{Rem}

\subsection{Discrete series and tempered representations}
\label{subs:discandtemp}
In this section we go over some of the important results concerning the theta lifts of discrete series and tempered representations. First, we recall the main results of Mui\'{c} \cite{muic2008structure} (Theorems 6.1 and 6.2), which give a complete description of theta lifts for discrete series representations, along with an insight into the structure of the full theta lift.
\begin{Thm}[6.1 and 6.2 in \cite{muic2008structure}]
\label{thm:Muic_disc}
Let $\sigma \in \Irr(G_n)$ be a discrete series representation. Set
\[
m_{\text{temp}}(\sigma)=
\begin{cases}
m(\sigma), &m(\sigma) > n + 1 + \epsilon\\
n + 1 + \epsilon, &m(\sigma) \leqslant n + 1 + \epsilon.
\end{cases}
\]
Then
\begin{enumerate}[(i)]
\item $\Theta(\sigma,m)$ is an irreducible tempered representation for $m(\sigma) \leqslant m \leqslant m_{\text{temp}}(\sigma)$.
\item If $m > m_{\text{temp}}(\sigma)$, then $\theta(\sigma, m)$ is the unique  irreducible (Langlands) quotient of
\[
\chi_W|\cdot|^{\frac{m-n-\epsilon - 1}{2}} \times \dotsc \times \chi_W|\cdot|^{\frac{m_{\text{temp}}(\sigma) - n - \epsilon + 1}{2}} \rtimes \theta(\sigma, m_{\text{temp}}(\sigma)).
\]
The remaining subquotients of $\Theta(\sigma,m)$ are either tempered, or equal to the Langlands quotient of
\[
\chi_W|\cdot|^{\frac{m-n-\epsilon - 1}{2}} \times \dotsc \times \chi_W|\cdot|^{\frac{m_1 - n -\epsilon + 1}{2}} \rtimes \sigma(m_1),
\]
where $\sigma(m_1)$ is a tempered irreducible subquotient of $\Theta(\sigma,m_1)$ for some $m > m_1 \geqslant m_{\text{temp}}(\sigma)$.
\end{enumerate}
\end{Thm}
Note that the recent results of Atobe and Gan \cite{atobe2017local} on theta lifts of tempered representations subsume most of the aforeknown results on the lifts of discrete series. For the sake of brevity, we do not state the relevant theorems here; we shall however use them on more than one occasion in the following sections. For now, we state a useful auxiliary result concerning tempered representations \cite[Proposition 5.5, Lemma 6.4]{atobe2017local}:
\begin{Prop}
\label{prop:Lpackets}
Let $\pi \in \Irr(G(W_n))$ be such that $\Theta(\pi,V_m) \neq 0$.
\begin{enumerate}[(1)]
\item If one of the following is satisfied
\begin{enumerate}[(i)]
\item $\pi$ is tempered and $m \leqslant n+1+\epsilon$;
\item $\pi$ is in discrete series and $\Theta(\pi,V_m)$ is the first lift to the going-up tower,
\end{enumerate}
then all the irreducible subquotients of $\Theta(\pi,V_m)$ are tempered.
\item If all the irreducible subquotients of $\Theta(\pi,V_m)$ are tempered, then they all belong to the same $L$-packet.
\end{enumerate}
\end{Prop}

\section{First occurrence}
\label{sec:first}
In this section we describe the first occurrence index of a $\chi$-generic representation $\pi \in \Irr(G(W_n))$. We fix $\epsilon = \pm 1$; if $\epsilon = 1$ we assume that a pair of target Witt towers $\mathcal{V}^+, \mathcal{V}^-$ is fixed. Recall that $m^{\text{down}}(\pi)$ denotes the lower of the two possible first occurrence indices. We set\footnote{This notation is motivated by the one used by Atobe and Gan, but does not have quite the same meaning as in the original paper \cite{atobe2017local}.}
\[
l(\pi) = n + \epsilon - m^{\text{down}}(\pi).
\]
By the standard module conjecture, $\pi$ is isomorphic to its standard module:
\[
\pi \cong \chi_V\delta_r\nu^{s_r} \times \dotsm \times \chi_V\delta_1\nu^{s_1} \rtimes \pi_0,
\]
where $\delta_i \in \Irr(\GL_{n_i}(\F))$ ($i = 1,\dots,r$) are discrete series representations, $s_r\geqslant\dotsb\geqslant s_1 > 0$, and $\pi_0 \in \Irr (G(W_{n_0}))$ is tempered. Note that $\pi_0$ is also $\chi$-generic by the hereditary property.

The following theorem determines the first occurrence of $\pi$. Note that the proof does not use the generic property; therefore, the theorem describes the first occurrence index for a much wider class of representations---those isomorphic to their standard modules.
\begin{Thm}
\label{thm:first_occ}
Let $\pi \in \Irr(G(W_n))$ be a representation isomorphic to its standard module,
\[
\pi \cong \chi_V\delta_r\nu^{s_r} \times \dotsm \times \chi_V\delta_1\nu^{s_1} \rtimes \pi_0.
\]
Then $l(\pi) = l(\pi_0)$.
\end{Thm}
Since the first occurrence of tempered representations is described by \cite[Theorem 4.1]{atobe2017local}, so that $l(\pi_0)$ is known, this is enough to infer the first occurrence index of $\pi$. For $\chi$-generic representations, we can say a bit more: the alternating property of Theorem 4.1 in \cite{atobe2017local} is never satisfied. To verify this, it suffices to temporarily switch the choice of the Whittaker datum needed for the Local Langlands Correspondence from the one fixed in \S \ref{subs:LLC} to the one corresponding to $\chi$; naturally this does not affect the first occurence index. Desideratum 1 of \cite{atobe2016uniqueness} then shows that the character $\eta$ in the Langlands parameter of $\pi_0$ is trivial, so there is no alternation. Alternatively, to describe $\eta$ and verify that there is no alternation, but without switching the Whittaker datum, one may use the results of \cite{kaletha2014genericity}. In short, since the alternating condition fails, there are only two possibilities:
\[
l(\pi_0) =
\begin{cases}
1, &\text{if }\phi\text{ contains } \chi_V;\\
-1, &\text{otherwise}
\end{cases}
\]
Of course, this only gives us $m^{\text{down}}(\pi)$, but we can get $m^{\text{up}}(\pi)$ using the conservation relation.

\begin{Rem}
\label{rem:notation}
Before proving the theorem, we remind the reader of the notation: recall that $\Theta_{l}(\pi) = \Theta(\pi,n+\epsilon-l)$. Combined with our definition of $l(\pi)$ and the conservation relation, this means that $\Theta_{l}(\pi)$ denotes the first non-zero lift of $\pi$ precisely when
\[
l =
\begin{cases}
l(\pi), \quad &\text{in the going-down tower};\\
-l(\pi)-2, \quad &\text{in the going-up tower}.
\end{cases}
\]
\end{Rem}

\begin{proof}[Proof of Theorem \ref{thm:first_occ}]
We first consider the going-up tower with respect to $\pi_0$. We compute $\Theta_{l}(\pi)$ with $l = -l(\pi_0)$. Since $s_i > 0$, we know that $\delta_i\nu^{s_i} \neq \textnormal{St}_k\nu^\frac{l-k}{2}$ for all $i = 1, \dotsc, r$. This allows us to use Corollary \ref{cor:theta_epi}; repeatedly applying it to
\[
\label{loc31:epi}\tag{\textasteriskcentered}
\chi_V\delta_r\nu^{s_r} \times \dotsm \times \chi_V\delta_1\nu^{s_1} \rtimes \pi_0 \twoheadrightarrow \pi
\]
we get
\[
\chi_W\delta_r\nu^{s_r} \times \dotsm \times \chi_W\delta_1\nu^{s_1} \rtimes \Theta_l(\pi_0) \twoheadrightarrow \Theta_l(\pi).
\]
This being the going-up tower, we have $\Theta_l(\pi_0) = 0$ (see Remark \ref{rem:notation}). Since the above map is surjective, this implies $\Theta_l(\pi) = 0$. We deduce that
\begin{enumerate}[(i)]
\item the going-up tower for $\pi$ is the same as for $\pi_0$ (since $\Theta_l(\pi) = 0$ with $l<0$, this must be the going-up tower for $\pi$);
\item we have $-l(\pi) \leqslant l$, i.e.~$l(\pi) \geqslant l(\pi_0)$.
\end{enumerate}
Now set $l = l(\pi_0)+2$; this time we consider the going-down tower with respect to $\pi_0$. We repeat the above argument to show that
\[
\chi_W\delta_r\nu^{s_r} \times \dotsm \times \chi_W\delta_1\nu^{s_1} \rtimes \Theta_l(\pi_0) \twoheadrightarrow \Theta_l(\pi).
\]
In this case it can happen that for some $i$ we have $\delta_i\nu^{s_i} = \textnormal{St}_k\nu^\frac{l-k}{2}$. To justify the use of Corollary \ref{cor:theta_epi}, we make the following observation: since $\chi_V\delta_r\nu^{s_r} \times \dotsm \times \chi_V\delta_1\nu^{s_1} \rtimes \pi_0$ is irreducible, we have
\[
\chi_V\delta_r\nu^{s_r} \times \dotsm \times \chi_V\delta_i\nu^{s_i}  \times\dotsm \times \chi_V\delta_1\nu^{s_1} \rtimes \pi_0 \cong \chi_V\delta_r\nu^{s_r} \times \dotsm \times \chi_V\delta_i^\vee\nu^{-s_i}  \times\dotsm \times \chi_V\delta_1\nu^{s_1} \rtimes \pi_0.
\]
Indeed, the irreducibility guarantees that $\chi_V\delta_i\nu^{s_i}$ and $\chi_V\delta_i^\vee\nu^{-s_i}$ may freely switch places with the rest of the factors, and that $\chi_V\delta_i\nu^{s_i} \rtimes \pi_0 \cong \chi_V\delta_i^\vee\nu^{-s_i} \rtimes \pi_0$.
This means that we can replace $\delta_i\nu^{s_i}$ with $\delta_i^\vee\nu^{-s_i} $ in \eqref{loc31:epi} and thus bypass the restriction of Corollary \ref{cor:theta_epi}.
Since $l > l(\pi_0)$, we have $\Theta_l(\pi_0) = 0$, so the above map implies $\Theta_l(\pi) = 0$ (see Remark \ref{rem:notation}). This means that $l(\pi) < l$, i.e.~$l(\pi) \leqslant l(\pi_0)$.

Combining the two inequalities we get the desired result, $l(\pi) = l(\pi_0)$.
It is worth mentioning the following fact obtained in the proof (see (i) above): the going-up (resp.~going-down) tower for $\pi$ coincides with the going-up (resp.~going-down) tower for $\pi_0$.
\end{proof}
One useful application of Theorem \ref{thm:first_occ} is given by the following lemma, which we use in subsequent sections. Although the result in question is only a special case of the results of \cite{muic2005reducibility}, we sketch the proof to demonstrate how theta correspondence can be used to approach general representation-theoretic problems.
\begin{Lem}
\label{lem:no_ex}
Let $\pi_0 \in \Irr(G(W_{n_0}))$ be a tempered representation with parameter $\phi$ and $l(\pi_0)=l > 0$. Set 
$\delta = \langle |\cdot|^{\frac{l+1}{2}}, |\cdot|^{a} \rangle$ where $a\in \langle\frac{1-l}{2},\frac{l+1}{2}]$ such that $\frac{l+1}{2}-a \in \mathbb{Z}$. Then
\[
\chi_V\delta \rtimes \pi_0
\]
reduces. The above representation is also reducible if $a = \frac{1-l}{2}$ and the multiplicity of $\chi_VS_l$ in $\phi$ is odd.
\end{Lem}
\begin{proof}
We omit the proof of the second statement, which concerns $a = \frac{1-l}{2}$. It may be verified by the results of \cite{muic2005reducibility}. For the main part of the statement, consider the representation $\pi = L(\chi_V\delta;\pi_0)$. Starting from $\chi_V\delta \rtimes \pi_0 \twoheadrightarrow \pi$ and using $\langle |\cdot|^{\frac{l-1}{2}}, |\cdot|^{a} \rangle \times |\cdot|^{\frac{l+1}{2}}\twoheadrightarrow \delta$, we get
\[
\chi_V\langle |\cdot|^{\frac{l-1}{2}}, |\cdot|^{a} \rangle \times \chi_V|\cdot|^{\frac{l+1}{2}} \rtimes \pi_0 \twoheadrightarrow \pi
\]
(if $a=\frac{l+1}{2}$, we omit $\langle |\cdot|^{\frac{l-1}{2}}, |\cdot|^{a} \rangle$).
We now observe that the irreducible subquotient $\pi'$ of $\chi_V|\cdot|^{\frac{l+1}{2}} \rtimes \pi_0$ which participates in the above epimorphism is equal to $L(\chi_V|\cdot|^{\frac{l+1}{2}};\pi_0)$. Indeed, comparing the left-hand side of the above epimorphism with the standard module of $\pi$ shows that $\pi'$ is non-tempered and that the $\GL$-part of the standard module of $\pi'$ has cuspidal support $\chi_V|\cdot|^\frac{l+1}{2}$. The only irreducible subquotient of $\chi_V|\cdot|^{\frac{l+1}{2}}\rtimes \pi_0$ with this property is $L(\chi_V|\cdot|^{\frac{l+1}{2}};\pi_0)$. Thus $
\chi_V\langle |\cdot|^{\frac{l-1}{2}}, |\cdot|^{a} \rangle \rtimes \pi' \twoheadrightarrow \pi$
becomes
\[
\chi_V\langle |\cdot|^{\frac{l-1}{2}}, |\cdot|^{a} \rangle \rtimes L(\chi_V|\cdot|^{\frac{l+1}{2}};\pi_0) \twoheadrightarrow \pi
\]
Theorem 4.5 of \cite{atobe2017local} now shows that $\pi' = L(\chi_V|\cdot|^{\frac{l+1}{2}};\pi_0)$ is equal to $\theta_{-l-2}(\theta_l(\pi_0))$ (lifting to level $-l-2$ on the going-up tower). In particular, $\theta_{l+2}(\pi') = \theta_l(\pi_0) \neq 0$, so that $l(\pi') \geq l+2$. This implies, by the conservation relation, that $\Theta_{-l-2}(\pi') = 0$ on the going-up tower. By Corollary \ref{cor:theta_epi} we have
\[
\chi_W\langle |\cdot|^{\frac{l-1}{2}}, |\cdot|^{a} \rangle \rtimes \Theta_{-l-2}(\pi') \twoheadrightarrow \Theta_{-l-2}(\pi)
\]
on the going up tower. Since $\Theta_{-l-2}(\pi')=0$, this shows $\Theta_{-l-2}(\pi)=0$.
 
Now assume that $\chi_V\delta \rtimes \pi_0$ is irreducible. Theorem \ref{thm:first_occ} then shows that $l(\pi) = l$. By the conservation relation, this implies $\theta_{-l-2}(\pi)$ is the first non-zero lift on the going-up tower. This contradicts $\Theta_{-l-2}(\pi)=0$; therefore, $\chi_V\delta \rtimes \pi_0$ cannot be irreducible.
\end{proof}

\section{The lifts}
\label{sec:lifts}
We are now ready to state the main result of this paper. The following theorem fully describes the theta lifts of a $\chi$-generic irreducible representation of $\Irr(G(W_n))$.

\begin{Thm}
\label{thm:lifts}
Let $\pi$ be an irreducible $\chi$-generic representation of $G(W_n)$ isomorphic to its standard module,
\[
\chi_V\delta_r\nu^{s_r} \times \dotsm \times \chi_V\delta_1\nu^{s_1} \rtimes \pi_0.
\]
Let $l$ be an odd integer such that $\theta_{l}(\pi) \neq 0$. Then
\[
\chi_W\delta_r\nu^{s_r} \times \dotsm \times \chi_W\delta_1\nu^{s_1} \rtimes \theta_{l}(\pi_0) \twoheadrightarrow \theta_{l}(\pi).
\]
Furthermore, if $\theta_{l}(\pi_0) = L(\chi_W\delta_k'\nu^{s_k'} \times \dotsm \times \chi_W\delta_1'\nu^{s_1'} \rtimes \tau)$, then $\theta_{l}(\pi)$ is uniquely determined by
\[
\theta_{l}(\pi) = L(\chi_W\delta_r\nu^{s_r}, \dotsc, \chi_W\delta_1\nu^{s_1}, \chi_W\delta_k'\nu^{s_k'}, \dotsc, \chi_W\delta_1'\nu^{s_1'}; \tau).
\]
\end{Thm}
To sketch our general approach, we now prove this theorem in case when $\theta_{l}(\pi_0)$ is tempered. The rest of the proof is more involved and will be given in several steps in Section \ref{sec:higher_lifts}.
\begin{proof}[Proof (when $\theta_l(\pi_0)$ is temepered)]
Theorem \ref{thm:first_occ} shows that we need only consider $\theta_{l}(\pi)$ for $l \leqslant 1$. With this in mind, Theorems 4.3 and 4.5 of \cite{atobe2017local} imply that the only cases in which $\theta_{l}(\pi_0)$ is tempered are the following: $l=1$ (in the going-down tower), $l = -1$, and $l=-3$ in the going-up tower when the multiplicity of $\chi_V$ in $\phi$ is odd (recall that $\phi$ denotes the parameter of $\pi_0$). We treat each of them separately.

\medskip

\noindent \underline{Case $1$: $l(\pi) = l(\pi_0) = -1$}\\
\noindent We recall the proof of Proposition C.4 in \cite{Gan2014}. In this case the first lift on both towers appears on the level $l = -1$. Since the left-hand side of
\[
\chi_V\delta_r\nu^{s_r} \times \dotsm \times \chi_V\delta_1\nu^{s_1} \rtimes \pi_0 \twoheadrightarrow \pi
\]
has a unique irreducible quotient, we can repeatedly apply Corollary \ref{cor:theta_epi} to arrive at
\[
\label{eq:first_lifts_epi1}\tag{1}
\chi_W\delta_r\nu^{s_r} \times \dotsm \times \chi_W\delta_1\nu^{s_1} \rtimes \Theta_{-1}(\pi_0) \twoheadrightarrow \Theta_{-1}(\pi).
\]
The use of Corollary \ref{cor:theta_epi} is justified: since $l < 0$ and $s_i > 0$, none of $\delta_i\nu^{s_i}$ are defined by a segment ending in $|\cdot|^\frac{l-1}{2}$.

Notice that $\Theta_{-1}(\pi_0)$ is irreducible and tempered: since $\pi_0$ is tempered, by Lemma \ref{lemma:temp_supp} there are discrete series representations $\delta_1', \dotsc, \delta_k',\pi_{00}$ such that
\[
\chi_V\delta_1' \times \dotsm \times \chi_V\delta_k'\rtimes \pi_{00} \twoheadrightarrow \pi_0
\]
(moreover, the left-hand side is completely reducible, by Lemma \ref{lemma:goldberg}). In this situation we can also use Corollary \ref{cor:theta_epi}: the segment defining $\delta_i'$ cannot end in $|\cdot|^\frac{l-1}{2}$ because $l < 0$. We get
\[
\chi_W\delta_1' \times \dotsm \times \chi_W\delta_k'\rtimes \Theta_{-1}(\pi_{00}) \twoheadrightarrow \Theta_{-1}(\pi_0).
\]
We can now use Theorem \ref{thm:Muic_disc}: $\Theta_{-1}(\pi_{00})$ is irreducible and tempered. This shows, by Lemma \ref{lemma:goldberg}, that the left-hand side in the above map is completely reducible, and that all of its irreducible subquotients are tempered. Thus, the same must hold for $\Theta_{-1}(\pi_0)$. Since $\Theta_{-1}(\pi_0)$ has a unique irreducible quotient, complete reducibility implies that $\Theta_{-1}(\pi_0)$ is itself irreducible (and tempered).

This shows that the left-hand side of \eqref{eq:first_lifts_epi1} is a standard module. Furthermore, since $\Theta_{-1}(\pi)\twoheadrightarrow \theta_{-1}(\pi)$, we can write
\[
\chi_W\delta_r\nu^{s_r} \times \dotsm \times \chi_W\delta_1\nu^{s_1} \rtimes \Theta_{-1}(\pi_0) \twoheadrightarrow \theta_{-1}(\pi)
\]
instead of \eqref{eq:first_lifts_epi1} and in this way arrive at the standard module for $\theta_{-1}(\pi)$.

\medskip

\noindent \underline{Case $2$: $l(\pi) = l(\pi_0) = 1$, $m_{\phi}(\chi_V)$ is odd, going-up tower}\\
This case is treated just like the previous one. Using Corollary \ref{cor:theta_epi} we get
\[
\chi_W\delta_r\nu^{s_r} \times \dotsm \times \chi_W\delta_1\nu^{s_1} \rtimes \Theta_{-3}(\pi_0) \twoheadrightarrow \Theta_{-3}(\pi)
\]
and it only remains to show that $\Theta_{-3}(\pi_0)$ is irreducible and tempered. The key point here is that, since $m_{\phi}(\chi_V)$ is odd, the parameter of $\pi_{00}$ (the representation appearing in the tempered support of $\pi_0$) also contains $\chi_V$. This implies $l(\pi_{00}) = 1$, which means that $\Theta_{-3}(\pi_{00})$ is the first lift of $\pi_{00}$ to the going-up tower. By Theorem \ref{thm:Muic_disc}, this means that $\Theta_{-3}(\pi_{00})$ is irreducible and tempered, so we can deduce the same properties for $\Theta_{-3}(\pi_{0})$ just as in case 1.

\medskip

\noindent \underline{Case $3$: $l(\pi) = l(\pi_0) = 1$, going-down tower}\\
As in the previous cases, Corollary \ref{cor:theta_epi} yields
\[
\chi_W\delta_r\nu^{s_r} \times \dotsm \times \chi_W\delta_1\nu^{s_1} \rtimes \Theta_{1}(\pi_0) \twoheadrightarrow \Theta_{1}(\pi)
\]
Note that with $l=1$ we still have $\delta_i\nu^{s_i} \neq \textnormal{St}_k\nu^\frac{l-k}{2}$ since $s_i > 0$ implies that the segment defining $\delta_i\nu^{s_i}$ cannot end in the trivial character $\mathbb{1} = |\cdot|^\frac{l-1}{2}$.

The difference is that in this case we do not know if $\Theta_{1}(\pi_0)$ is irreducible. However, we do know that all of its subquotients are tempered (and that they belong to the same L-packet, by Proposition \ref{prop:Lpackets}).
Therefore, there is a tempered irreducible subquotient $\sigma$ of 
$\Theta_{1}(\pi_0)$ such that
\[
\chi_W\delta_r\nu^{s_r} \times \dotsm \times \chi_W\delta_1\nu^{s_1} \rtimes \sigma \twoheadrightarrow \theta_{1}(\pi),
\]
and we need to show that $\sigma \cong \theta_{1}(\pi_0)$.

To this end, we use Corollary \ref{cor:theta_epi} again, this time in the opposite direction, and with $l=-1$. Having in mind that $\theta_{-1}(\theta_{1}(\pi)) = \pi$, we get
\[
\chi_V\delta_r\nu^{s_r} \times \dotsm \times \chi_V\delta_1\nu^{s_1} \rtimes \Theta_{-1}(\sigma) \twoheadrightarrow \pi
\]
Just like in case 1, $\Theta_{-1}(\sigma)$ is irreducible and tempered, so the left-hand side of the above epimorphism is in fact the standard module of $\pi$. The uniqueness of the standard module now implies $\Theta_{-1}(\sigma) = \theta_{-1}(\sigma) = \pi_0$, that is, $\sigma \cong \theta_{1}(\pi_0)$.
\end{proof}
The use of Corollary \ref{cor:theta_epi} in both directions (just like in case 3) is the starting point of our approach to determining the higher theta lifts.

\section{Interlude: irreducibility in \texorpdfstring{$\GL_n(\F)$}{GLn(F)}}
\label{sec:interlude}
Before advancing to the main part of the proof of Theorem \ref{thm:lifts}, we review some auxiliary results concerning certain induced representations of $\GL_n(\F)$ which appear in our calculations. The reader is advised to skim through this section at first reading, since only the statements (and not their proofs) are crucial for the next section.

We recall the work of Zelevinsky \cite{zelevinsky1980induced}. Any (unitary) cuspidal representation $\rho$ and $a,b \in \mathbb{R}$ such that $b-a \in \mathbb{Z}_{\geq 0}$ define a segment $[\nu^a\rho, \nu^b\rho]$ of irreducible cuspidal representations. To this segment, we attach the induced representation
\[
\nu^b\rho \times \nu^{b-1}\rho \times \dotsb \times \nu^a\rho.
\]
A representation of this form has a unique (Langlands) irreducible quotient, which we denote by $\langle \nu^a\rho, \nu^b\rho\rangle$. It also possesses a unique irreducible subrepresentation, denoted by $\langle \nu^b\rho, \nu^a\rho\rangle$. A representation of the form $\langle \nu^b\rho, \nu^a\rho\rangle$ (with $b\geq a$) is essentially square integrable; conversely, any essentially square integrable representation of the general linear group can be obtained in this way from a (uniquely determined) segment. At various points of this section, we freely use the terminology and results of \cite{zelevinsky1980induced} on linked segments.

We start by examining the relation between $\langle \nu^a\rho, \nu^b\rho\rangle$ with $a\leq b$ and $\langle \nu^d\rho, \nu^c\rho\rangle$ with $c\leq d$. We say that the segments $[\nu^a\rho,\nu^b\rho]$ and $[\nu^{c}\rho',\nu^{d}\rho']$ are adjacent if $\rho = \rho'$, and $a = d+1$ or $c = b+1$.
\begin{Lem}
\label{lemma:irr_zel}
Assume that the segments $[\nu^a\rho,\nu^b\rho]$ and $[\nu^{c}\rho',\nu^{d}\rho']$ are not adjacent. Then
\[
\langle \nu^a\rho, \nu^b\rho\rangle \times \langle \nu^d\rho', \nu^c\rho'\rangle \quad \text{and} \quad \langle \nu^d\rho', \nu^c\rho'\rangle\times \langle \nu^a\rho, \nu^b\rho\rangle
\]
are irreducible and isomorphic.
\end{Lem}\begin{proof}
This is a special case of Lemma I.6.3 of \cite{Moeglin1989}.
\end{proof}
We now examine the case when the segments are adjacent. 
In what follows, we assume that $\rho$ is always equal to the trivial character $\mathbb{1}$ of $\GL_1(\F)$, although the same proofs work for any cuspidal $\rho$. 
\begin{Lem}
\label{rem:swap2}
Let $[a,b]$ and $[b+1,d]$ be adjacent segments. Then the representation
\[
\langle |\cdot|^{b+1}, |\cdot|^d\rangle \times \langle |\cdot|^b, |\cdot|^a\rangle
\]
is of length two. Furthermore,
\begin{itemize}
\item its unique irreducible quotient is the Langlands quotient of
\[
|\cdot|^d \times |\cdot|^{d-1} \times \dotsm \times  |\cdot|^{b+1} \times \langle |\cdot|^b, |\cdot|^a\rangle;
\]
it is also the unique irreducible quotient of $\langle |\cdot|^{b-1},|\cdot|^a\rangle\times\langle |\cdot|^b,|\cdot|^d\rangle$.
\item Its unique irreducible subrepresentation is the Langlands quotient of
\[
|\cdot|^d \times |\cdot|^{d-1} \times \dotsm \times  |\cdot|^{b+2} \times \langle |\cdot|^{b+1}, |\cdot|^a\rangle
\]
at the same time, it is the unique irreducible quotient of $\langle |\cdot|^b, |\cdot|^a\rangle\times\langle |\cdot|^{b+1}, |\cdot|^d\rangle$.
\end{itemize}
\end{Lem}
\begin{proof}
In this case, the representation $\pi =\langle |\cdot|^{b+1}, |\cdot|^d\rangle \times \langle |\cdot|^b, |\cdot|^a\rangle$ reduces. Moreover, formula \eqref{eq_Hopf} shows that $R_{P_1}(\pi)$ is of length two, so $\pi$ is also of length two. We leave the proof of the rest of the lemma to the reader.
\end{proof}
We also make use of the following lemma. If $[b,d]$ and $[d+1,e]$ are adjacent segments, the above lemma shows that $\langle |\cdot|^{d+1}, |\cdot|^e\rangle \times \langle |\cdot|^{d}, |\cdot|^b\rangle$ has a unique irreducible quotient, which we now denote by $L$.
\begin{Lem}
\label{lemma_specific_reducibility}
\begin{enumerate}[(i)]
\item For any $a \leq b$, the representation $L \times \langle |\cdot|^{d}, |\cdot|^a\rangle$ is irreducible. In particular, $L \times \langle |\cdot|^{d}, |\cdot|^a\rangle = \langle |\cdot|^{d}, |\cdot|^a\rangle \times L$.
\item For any $c$ such that $d-1 \geq c \geq b$, the representation $L \times \langle |\cdot|^{d-1}, |\cdot|^c\rangle$ is irreducible. In particular, $L \times \langle |\cdot|^{d-1}, |\cdot|^c\rangle =\langle |\cdot|^{d-1}, |\cdot|^c\rangle \times L$.
\end{enumerate}
\end{Lem}

\begin{proof}
We first prove (ii). Denote by $\xi$ the unique (Langlands) quotient of
\[
\langle |\cdot|^{d+1}, |\cdot|^e\rangle \times \langle |\cdot|^{d}, |\cdot|^b\rangle \times \langle |\cdot|^{d-1}, |\cdot|^c\rangle.
\]
Since $L \times \langle |\cdot|^{d-1}, |\cdot|^c\rangle$ is a quotient of the above, and $\xi$ is the unique irreducible quotient, we also have $L \times \langle |\cdot|^{d-1}, |\cdot|^c\rangle \twoheadrightarrow \xi$. We now have
\begin{align*}
\langle |\cdot|^{d+1}, |\cdot|^e\rangle \times \langle |\cdot|^{d}, |\cdot|^b\rangle \times \langle |\cdot|^{d-1}, |\cdot|^c\rangle &\cong \langle |\cdot|^{d+1}, |\cdot|^e\rangle \times \langle |\cdot|^{d-1}, |\cdot|^c\rangle \times \langle |\cdot|^{d}, |\cdot|^b\rangle\\
&\cong  \langle |\cdot|^{d-1}, |\cdot|^c\rangle \times \langle |\cdot|^{d+1}, |\cdot|^e\rangle \times \langle |\cdot|^{d}, |\cdot|^b\rangle,
\end{align*}
where the first isomorphism follows from the fact that $[b,d]$ contains $[c,d-1]$ (so they are not linked), and the second from Lemma \ref{lemma:irr_zel}.
Since $\xi$ is the unique quotient of the above representation, it must also be a quotient of $ \langle |\cdot|^{d-1}, |\cdot|^c\rangle \times L$. On the other hand, $\xi$ is (by definition) a subrepresentation of $ \langle |\cdot|^{d-1}, |\cdot|^c\rangle  \times L$. Since it appears with multiplicity one ($\xi$ being the Langlands quotient), it follows that $ \langle |\cdot|^{d-1}, |\cdot|^c\rangle \times L$ and $L \times \langle |\cdot|^{d-1}, |\cdot|^c\rangle$ are irreducible and isomorphic.

We now prove (i) in a similar manner. Let $\xi$ be the unique (Langlands) quotient of
\[
 \langle |\cdot|^{d+1}, |\cdot|^e \rangle \times  \langle |\cdot|^{d}, |\cdot|^b\rangle \times \langle |\cdot|^{d}, |\cdot|^a\rangle.
\]
As in (ii), we also have $L \times \langle |\cdot|^{d}, |\cdot|^a\rangle \twoheadrightarrow \xi$. By Lemma \ref{rem:swap2}, this implies
\[
 \langle |\cdot|^{d-1}, |\cdot|^b \rangle \times  \langle |\cdot|^{d}, |\cdot|^e \rangle \times  \langle |\cdot|^{d}, |\cdot|^a \rangle \twoheadrightarrow \xi.
\]
We now have
\begin{align*}
\langle |\cdot|^{d-1}, |\cdot|^b \rangle \times  \langle |\cdot|^{d}, |\cdot|^e \rangle \times  \langle |\cdot|^{d}, |\cdot|^a \rangle &\cong \langle |\cdot|^{d-1}, |\cdot|^b \rangle \times  \langle |\cdot|^{d}, |\cdot|^a \rangle \times  \langle |\cdot|^{d}, |\cdot|^e \rangle  \\
&\cong \langle |\cdot|^{d}, |\cdot|^a \rangle  \times  \langle |\cdot|^{d-1}, |\cdot|^b \rangle \times  \langle |\cdot|^{d}, |\cdot|^e \rangle,
\end{align*}
where the first isomorphism follows from Lemma \ref{lemma:irr_zel}, and the second from the fact that $[a,d]$ contains $[b,d-1]$ (so they are not linked).
Therefore, $\xi$ is an irreducible quotient of $\langle |\cdot|^{d}, |\cdot|^a \rangle  \times  \langle |\cdot|^{d-1}, |\cdot|^b \rangle \times  \langle |\cdot|^{d}, |\cdot|^e \rangle$, only this time, we do not know if $\xi$ is unique. Applying Lemma \ref{rem:swap2} to $\langle |\cdot|^{d-1}, |\cdot|^b \rangle \times  \langle |\cdot|^{d}, |\cdot|^e \rangle$ now shows that $\xi$ is necessarily a quotient of $ \langle |\cdot|^{d}, |\cdot|^a \rangle \times L$ or of $ \langle |\cdot|^{d}, |\cdot|^a \rangle \times \langle |\cdot|^{d}, |\cdot|^e \rangle \times \langle |\cdot|^{d-1}, |\cdot|^b \rangle$. If the latter were true, we would have
\[
\langle |\cdot|^{d}, |\cdot|^e \rangle \times \langle |\cdot|^{d}, |\cdot|^a \rangle \times \langle |\cdot|^{d-1}, |\cdot|^b \rangle \cong \langle |\cdot|^{d}, |\cdot|^a \rangle \times \langle |\cdot|^{d}, |\cdot|^e \rangle \times \langle |\cdot|^{d-1}, |\cdot|^b \rangle \twoheadrightarrow \xi,
\]
contradicting the shape of the standard module for $\xi$. Thus, $\xi$ is a quotient of $\langle |\cdot|^{d}, |\cdot|^a \rangle  \times L$. As we already know $L \times \langle |\cdot|^{d}, |\cdot|^a \rangle  \twoheadrightarrow \xi$, and $\xi$ appears with multiplicity one (again, using the multiplicity one property of the Langlands quotient), the conclusion follows.
\end{proof}
\section{Higher lifts}
\label{sec:higher_lifts}
We are now ready to prove the rest of Theorem \ref{thm:lifts}. Recall that we have already settled the cases in which $l \geq -1$ on the going down tower, as well as the first lifts on the going-up tower when $\theta_l(\pi_0)$ is tempered.
In all the remaining cases $l = n+\epsilon-m$ is negative (and odd), so we adjust the notation: letting $l > 0$ be an arbitrary odd integer, we want to determine $\theta_{-l}(\pi)$.

\subsection{Subquotients of \texorpdfstring{$\Theta(\pi_0)$}{Theta(pi0)}}
\label{subs:subq}
We fix $l>0$ odd and set $\sigma = \theta_{-l}(\pi)$; our goal is to determine $\sigma$. Since $\pi \in \Irr(G(W_n))$ is $\chi$-generic, it is isomorphic to its standard module:
\[
\pi \cong \chi_V\delta_r\nu^{s_r} \times \dotsm \times \chi_V\delta_1\nu^{s_1} \rtimes \pi_0.
\]
Applying Corollary \ref{cor:theta_epi} just like in Section \ref{sec:lifts}, we get
\[
\label{eq:epi0}\tag{0}
\chi_W\delta_r\nu^{s_r} \times \dotsm \times \chi_W\delta_1\nu^{s_1} \rtimes \Theta_{-l}(\pi_0)\twoheadrightarrow \Theta_{-l}(\pi)\twoheadrightarrow \theta_{-l}(\pi) = \sigma.
\]
Our main task is to determine the irreducible subquotient of $\Theta_{-l}(\pi_0)$ which participates in the above epimorphism. To describe it, we need to further analyze $\pi_0$. By Lemma \ref{lemma:temp_supp} we can write
\[
\chi_V\delta_1' \times \dotsm \times \chi_V\delta_k'\rtimes \pi_{00} \twoheadrightarrow \pi_0
\]
where $\delta_1', \dotsc, \delta_k',\pi_{00}$ are irreducible discrete series representations. Setting $\Delta = \delta_1' \times \dotsm \times \delta_k'$ and applying Corollary \ref{cor:theta_epi} again, we get
\[
\chi_W\Delta \rtimes \Theta_{-l}(\pi_{00}) \twoheadrightarrow \Theta_{-l}(\pi_0).
\]
Thus, we can write
\[
\label{eq:epi1}\tag{1}
\chi_W\delta_r\nu^{s_r} \times \dotsm \times \chi_W\delta_1\nu^{s_1} \times \chi_W\Delta \rtimes \Theta_{-l}(\pi_{00}) \twoheadrightarrow \sigma.
\]
We would now like to identify the subquotient (call it $\sigma_0$) of $\Theta_{-l}(\pi_{00})$ which participates in the above epimorphism. 
We will show the following:
\begin{Prop}
\label{prop:main}
The subquotient of $\Theta_{-l}(\pi_{00})$ which participates in $\eqref{eq:epi1}$ is equal to $\theta_{-l}(\pi_{00})$.
\end{Prop}
\begin{Rem}
\label{rem:possible_subq}
We recall the results of \cite{muic2008structure} and \cite{atobe2017local} (see Theorem \ref{thm:Muic_disc}): $\theta_{-l}(\pi_{00})$ is the Langlands quotient of
\[
\label{eq:firstlift}\tag{\textasteriskcentered}
\chi_W|\cdot|^\frac{l-1}{2} \times \chi_W|\cdot|^\frac{l-3}{2} \times \dotsm \times \chi_W|\cdot|^\frac{1+l_0}{2} \rtimes \theta_{-l_0}(\pi_{00})
\]
where $l_0 = \min\{l > 0 : \theta_{-l}(\pi_{00})\neq 0\}$. Any other irreducible subquotient of $\Theta_{-l}(\pi_{00})$ is either:
\begin{itemize}
    \item tempered; or
    \item the Langlands quotient of
\[
\chi_W|\cdot|^\frac{l-1}{2} \times \chi_W|\cdot|^\frac{l-3}{2} \times \dotsm \times \chi_W|\cdot|^\frac{1+l'}{2} \rtimes \sigma_0',
\]
where $\sigma_0'$ is a tempered subquotient of $\Theta_{-l'}(\pi_{00})$ for some $l'$, with $l > l' \geqslant l_0$.
\end{itemize}
Note that the above Langlands quotient is also the unique quotient of $\chi_W\langle|\cdot|^{\frac{1+l'}{2}},|\cdot|^{\frac{l-1}{2}}\rangle \rtimes \sigma_0'$.
\end{Rem}
We are now ready to prove the proposition.
\begin{proof}[Proof of Proposition \ref{prop:main}]
Assume, with the above remark in mind, that the subquotient of $\Theta_{-l}(\pi_{00})$ we want to find (and which we denote by $\sigma_{0}$) is isomorphic to the unique irreducible quotient of
\[
\chi_W\langle|\cdot|^\frac{1+l'}{2},|\cdot|^\frac{l-1}{2}\rangle \rtimes \sigma_0'.
\]
Here we allow the segment $[\frac{1+l'}{2},\frac{l-1}{2}]$ to be empty, i.e.~that $\sigma_{0} = \sigma_{0}'$ is tempered. We want to prove that $l' = l_0$, so that $\sigma_0$ is given by the quotient of \eqref{eq:firstlift} in the above remark.

As $\sigma_0$ participates in \eqref{eq:epi1}, we have
\[
\chi_W\Pi \times \chi_W\Delta \times\chi_W\langle|\cdot|^\frac{1+l'}{2},|\cdot|^\frac{l-1}{2}\rangle \rtimes \sigma_0' \twoheadrightarrow \sigma,
\]
where we used $\Pi$ to denote $\delta_r\nu^{s_r} \times \dotsm \times \delta_1\nu^{s_1}$. We are now in a situation which matches the requirements of Lemma \ref{lemma:irr_zel}: $\langle|\cdot|^\frac{1+l'}{2},|\cdot|^\frac{l-1}{2}\rangle$ can switch places with (almost) all of $\delta'_i$ which define $\Delta$. If all of $\delta_i'$ could be replaced, we would have
\[
\label{eq:caseI}
\tag{I}
\chi_W\Pi \times \chi_W\langle|\cdot|^\frac{1+l'}{2},|\cdot|^\frac{l-1}{2}\rangle \times \chi_W\Delta \times \sigma_0' \twoheadrightarrow \sigma.
\]
The only case in which we cannot proceed as above is the one in which $[\frac{1+l'}{2},\frac{l-1}{2}]$ is adjacent to the segment defining $\delta'_i$ for some $i \in \{1,\dotsc, k\}$, that is, when
\[
\delta'_i = \langle|\cdot|^\frac{l'-1}{2},|\cdot|^\frac{1-l'}{2}\rangle = \textnormal{St}_{l'}.
\]
This does not cause severe complications: without loss of generality we may assume that $\delta'_1,\delta'_2,\dotsc,\delta'_i,\dotsc,\delta'_k$ are ordered increasingly with respect to the length of the defining segments. We can apply Lemma \ref{lemma:irr_zel} to swap $\langle|\cdot|^\frac{1+l'}{2},|\cdot|^\frac{l-1}{2}\rangle$ with $\delta'_{i+1},\dotsc,\delta'_k$. After this, we arrive at the following situation:
\[
\dotsb \times \chi_W\delta'_i \times \chi_W\langle|\cdot|^\frac{1+l'}{2},|\cdot|^\frac{l-1}{2}\rangle \times \dotsb \twoheadrightarrow \sigma.
\]
Now, Lemma \ref{rem:swap2} implies that we have
\[
\dotsb\times \chi_W\langle|\cdot|^\frac{1+l'}{2},|\cdot|^\frac{l-1}{2}\rangle \times \chi_W\delta'_i \times \dotsb \twoheadrightarrow \sigma
\]
or
\[
\dotsb \times \chi_W\langle|\cdot|^\frac{3+l'}{2},|\cdot|^\frac{l-1}{2}\rangle \times \chi_W\langle|\cdot|^\frac{l'+1}{2},|\cdot|^\frac{1-l'}{2}\rangle \times \dotsb \twoheadrightarrow \sigma.
\]
The first case leads us to the same conclusion as in \eqref{eq:caseI}, whereas the second---having in mind that we can now swap $\langle|\cdot|^\frac{3+l'}{2},|\cdot|^\frac{l-1}{2}\rangle$ with all the $\delta'_1,\dotsc, \delta'_{i-1}$---leads to
\[
\label{eq:caseII}
\tag{II}
\chi_W\Pi \times \chi_W\langle|\cdot|^\frac{3+l'}{2},|\cdot|^\frac{l-1}{2}\rangle \times \chi_W\langle|\cdot|^\frac{l'+1}{2},|\cdot|^\frac{1-l'}{2}\rangle \times \chi_W\Delta' \rtimes \sigma_0' \twoheadrightarrow \sigma,
\]
where $\Delta' \cong \delta'_1 \times \dotsm \times \widehat{\delta'_i} \times \dotsm \times \delta'_k$ (here $\widehat{\delta'_i}$ signifies that we omit $\delta_i'$ from the product).

In both cases \ref{eq:caseI} and \ref{eq:caseII} we can do the same: using the results of \S \ref{sec:interlude}, $\Pi \times \langle|\cdot|^\frac{1+l'}{2},|\cdot|^\frac{l-1}{2}\rangle$ (resp.~$\Pi \times \langle|\cdot|^\frac{3+l'}{2},|\cdot|^\frac{l-1}{2}\rangle \times \langle|\cdot|^\frac{l'+1}{2},|\cdot|^\frac{1-l'}{2}\rangle$) in the epimorphism can be rearranged into
\[
\Pi' = \overline{\delta}_{t}\nu^{e_t}\times \dotsm \times \overline{\delta}_{1}\nu^{e_1},
\]
where $\overline{\delta}_i$ are irreducible discrete series representations and $e_t \geqslant \dotsb \geqslant e_1 > 0$). To explain why this can be achieved, we use an inductive argument. Let $\Pi =  \delta_{r}\nu^{s_r}\times \dotsm \times \delta_{1}\nu^{s_1}$ be a representation analogous to $\Pi'$ above. We prove that, if $\chi_W\Pi \times \chi_W\langle|\cdot|^c,|\cdot|^d\rangle \rtimes A \twoheadrightarrow \sigma$ (where $A$ is an irreducible representation and $c \leq d$), then this epimorphism can be rearranged into $\chi_W\Pi' \rtimes A \twoheadrightarrow \sigma$ in a suitable manner; here the cuspidal supports of $\Pi \times \langle|\cdot|^c,|\cdot|^d\rangle$ and $\Pi'$ coincide. We induce on $d-c+r$ (the length of the segment plus the number of $\GL$-representations to the left of the segment). To explain the induction step, assume that we have
\[
\dotsm \times \chi_W\delta\nu^s \times \chi_W\langle|\cdot|^c,|\cdot|^d\rangle \times \dotsm \to \sigma.
\]
If $d \leq s$, no further rearrangement is needed, and we may simply write $\chi_W|\cdot|^d \times \dotsm \times \chi_W|\cdot|^c$ instead of $\chi_W\langle|\cdot|^c,|\cdot|^d\rangle$ (recalling that the latter is a quotient of the former), and we are done. Otherwise, setting $s' = \min\{s'\in\{c,c+1,\dotsc,d\}: s \leq s' \leq d\}$, we may write (omitting $\chi_W\langle|\cdot|^c,|\cdot|^{s'-1}\rangle$ if $s' = c$)
\[
\dotsm \times \chi_W\delta\nu^s \times \chi_W\langle|\cdot|^{s'},|\cdot|^d\rangle \times \chi_W\langle|\cdot|^c,|\cdot|^{s'-1}\rangle \times \dotsm \to \sigma.
\]
We now use Lemma \ref{lemma:irr_zel} (or \ref{rem:swap2} if $\delta\nu^s = \langle|\cdot|^{c-1},|\cdot|^a\rangle$ for some $a < c$) to rewrite this as 
\[
\dotsm \times\chi_W\langle|\cdot|^{s'},|\cdot|^d\rangle  \times \chi_W\delta\nu^s \times  \chi_W\langle|\cdot|^c,|\cdot|^{s'-1}\rangle \times \dotsm \to \sigma
\]
or (in case $\delta\nu^s = \langle|\cdot|^{c-1},|\cdot|^a\rangle$) as
\[
\dotsm \times\chi_W\langle|\cdot|^{c+1},|\cdot|^d\rangle \times \chi_W\langle|\cdot|^{c},|\cdot|^a\rangle \times \dotsm \to \sigma.
\]
In both cases, the resulting segment $\langle|\cdot|^{s'},|\cdot|^d\rangle$ (or $\langle|\cdot|^{c+1},|\cdot|^d\rangle$) is shorter than the original segment $\langle|\cdot|^{c},|\cdot|^d\rangle$, so we can continue this procedure inductively. An exception to this is the situation $c > s$ when the resulting segment $\langle|\cdot|^{s'},|\cdot|^d\rangle$ is again equal to $\langle|\cdot|^{c},|\cdot|^d\rangle$. Then the segment we get is not shorter, but there are fewer $\GL$-representations appearing to the left of it, which again enables the induction step.

In short, we get a standard module
\[
\chi_W\Pi' \rtimes \tau \twoheadrightarrow \sigma
\]
for $\sigma$, where $\tau$ is an irreducible (and obviously tempered) subquotient of $\chi_W\Delta \rtimes \sigma_0'$ (in case \ref{eq:caseI}), or  $\chi_W\Delta' \rtimes \sigma_0'$ (in case \ref{eq:caseII}). This shows the following:
\begin{itemize}
    \item[(I)] In case \eqref{eq:caseI}, the cuspidal support of $\Pi'$ consists of $|\cdot|^\frac{l-1}{2},|\cdot|^\frac{l-3}{2},\dotsc,|\cdot|^\frac{l'+1}{2}$ in addition to the cuspidal support of $\Pi$.
    \item[(II)] In case \eqref{eq:caseII}, the cuspidal support of $\Pi'$ consists of $|\cdot|^\frac{l-1}{2},|\cdot|^\frac{l-3}{2},\dotsc,|\cdot|^\frac{l'+1}{2}$ and the segment $[|\cdot|^\frac{1-l'}{2},|\cdot|^\frac{l'-1}{2}]$ in addition to the cuspidal support of $\Pi$.
\end{itemize}
We now use Kudla's filtration to return to the $(W_n)$ tower: We want to get $\theta_l(\sigma)$ (while knowing that $\theta_l(\sigma) = \pi$) by repeated use of Corollary \ref{cor:shaving} on
\[
\chi_W\overline{\delta}_{t}\nu^{e_t}\times \dotsm \times \chi_W\overline{\delta}_{1}\nu^{e_1}\rtimes \tau \twoheadrightarrow \sigma.
\]
If we apply the corollary exactly $t$ times, we get
\[
\chi_V(\overline{\delta}_{t}\nu^{e_t})\times \dotsm \times \chi_V(\overline{\delta}_{1}\nu^{e_1}) \rtimes \Theta_{l-2k}(\tau) \twoheadrightarrow \pi.
\]
Here we use the notation $(\overline{\delta}\nu^{e})$ introduced in Remark \ref{rem:parentnotation}. Furthermore, $k$ denotes the number of segments on which option (ii) of Corollary \ref{cor:shaving} is used, which is why $\Theta_l$ becomes $\Theta_{l-2k}$.

Again we rearrange the representations $\chi_V(\overline{\delta}_i\nu^{e_i})$ in order to get a standard module (i.e.~so that the midpoints of the corresponding segments form a decreasing sequence).

We need to show that this is actually possible.
Each $\overline{\delta}_i\nu^{e_i} = \langle\rho_i\nu^{a_i+e_i},\rho_i\nu^{-a_i+e_i}\rangle$ is defined by a segment with midpoint $e_i > 0$ (here $a_i$ is a non-negative half-integer). After applying Corollary \ref{cor:shaving} and bringing them in front of $\Theta_{l-2k}(\tau)$, some of the $(\overline{\delta}_i\nu^{e_i})$ (namely, those obtained via option (ii)) are defined by slightly modified segments of the form $[\rho_i\nu^{-a_i+e_i}, \rho_i\nu^{a_i+e_i-1}]$, with midpoint $(e_i-\frac 12)$. We have the following possibilities:
\begin{itemize}
    \item If $a_i = 0$, then $[\rho_i\nu^{-a_i+e_i}, \rho_i\nu^{a_i+e_i-1}]$ is empty;
    therefore $(\overline{\delta}_i\nu^{e_i})$ doesn't exist.
    \item It is possible that $e_i-\frac 1 2=0$, i.e. that $0$ is the midpoint of the new segment.
    \item All the other segments satisfy $e_i - \frac 12> 0$: if we use option (ii) in Corollary \ref{cor:shaving}, this implies (among other things) that $e_i$ is a half-integer; in particular, $e_i \geqslant \frac 12$.
\end{itemize}
    Furthermore, note that we can really reorder the $(\overline{\delta}_i\nu^{e_i})$ to obtain a decreasing sequence of exponents. Namely, if this requires us to swap $(\overline{\delta}_{i+1}\nu^{e_{i+1}})$ and $(\overline{\delta}_i\nu^{e_i})$, this means the following: the ordering has changed
    because $\overline{\delta}_{i+1}\nu^{e_{i+1}}$ was obtained by means of option (ii), whereas option (i) was used on $\overline{\delta}_i\nu^{e_i}$---otherwise, they would still be ordered correctly.
    This implies
    \begin{align*}
   (\overline{\delta}_{i+1}\nu^{e_{i+1}}) &= \langle|\cdot|^{-a_{i+1}+e_{i+1}},|\cdot|^{a_{i+1}+e_{i+1}-1}\rangle\\
    (\overline{\delta}_i\nu^{e_i}) &=\langle\rho\nu^{a_{i}+e_{i}},\rho\nu^{-a_{i}+e_{i}}\rangle.
    \end{align*}
    If we assume that these segments are linked, then $\rho = \mathbb{1}$ and the following assertions hold:
    \begin{itemize}
    \item the segments are linked, so we have $e_i-e_{i+1} \in \frac 12\mathbb{Z}$;
    \item they need to be swapped, so $e_{i+1}-\frac 12 < e_i$;
    \item the original ordering implies $e_{i+1} \geqslant e_i$.
    \end{itemize}
    This is only possible if $e_i = e_{i+1}$. From here we easily deduce that the segments cannot really be linked, so they can freely switch places.

We have thus shown that the desired rearrangement is indeed possible. In short, we can write
\[
\chi_V\Pi'' \times \chi_V\Delta'' \rtimes \Theta_{l-2k}(\tau) \twoheadrightarrow \pi.
\]
Here $\Pi''\times \Delta''$ denotes the product of $(\overline{\delta}_i\nu^{e_i})$ (in decreasing order of $e_i$); here we have grouped all the segments of the form $\langle|\cdot|^{a},|\cdot|^{-a}\rangle$ into $\Delta''$.

Note that $l-2k > 0$: $\Theta_{l-2k}(\tau)$ indicates that $\frac{l+1}{2}-k$ was the endpoint of the last segment on which we used option (ii) of Corollary \ref{cor:shaving}; therefore $\frac{l+1}{2}-k>0$, i.e.~$l-2k>-1$ (in fact, $l-2k$ is odd so we have $l-2k>0$). Since $l-2k > 0$, Proposition \ref{prop:Lpackets} shows that all the subquotients of $\Theta_{l-2k}(\tau)$ are tempered. We thus see that the standard module of $\pi$ is equal to
\[
\chi_V\Pi'' \rtimes \pi_0'',
\]
where $\pi_0''$ is a (tempered) irreducible subquotient of $\chi_V\Delta'' \rtimes \Theta_{l-2k}(\tau)$. The uniqueness of the standard module now forces
\[
\chi_V\Pi'' = \chi_V\Pi \quad \text{and}\quad \pi_0'' \cong \pi_0.
\]
In particular, $\Pi''$ and $\Pi$ have the same cuspidal support. We have already compared the cuspidal supports of $\Pi$ and $\Pi'$. On the other hand, since option (ii) of Corollary \ref{cor:shaving} was applied exactly $k$ times, we see that, compared to $\Pi'$, the cuspidal support of $\Pi''$ is missng
\[
|\cdot|^\frac{l-1}{2}, |\cdot|^\frac{l-3}{2}, \dotsc, |\cdot|^{\frac{l+1}{2}-k},
\]
along with all the segments grouped into $\Delta''$. In other words:
\begin{align*}
S(\Pi') &= S(\Pi) \cup
\begin{cases}
\{|\cdot|^\frac{l-1}{2},|\cdot|^\frac{l-3}{2},\dotsc,|\cdot|^\frac{l'+1}{2}\}, &\quad \text{in case I},\\
 \{|\cdot|^\frac{l-1}{2},|\cdot|^\frac{l-3}{2},\dotsc,|\cdot|^\frac{l'+1}{2}\} \cup [|\cdot|^\frac{1-l'}{2},|\cdot|^\frac{l'-1}{2}], &\quad \text{in case II};
\end{cases}\\
S(\Pi') &= S(\Pi'') \cup
\{|\cdot|^\frac{l-1}{2}, |\cdot|^\frac{l-3}{2}, \dotsc, |\cdot|^{\frac{l+1}{2}-k}\} \cup \Delta'',
\end{align*}
where we have used $S(\Pi)$ to denote the cuspidal support of $\Pi$, i.e.~the multiset of all cuspidal representations of the $\GL$-blocks in the Levi factor from which $\Pi$ is induced.
Comparing the two expressions for $S(\Pi')$, we get
\begin{gather*}
\frac{l'+1}{2} = \frac{l+1}{2}-k \quad \text{and} \quad \Delta'' = 
\begin{cases}
\emptyset, &\qquad \text{in case I},\\
[|\cdot|^\frac{1-l'}{2},|\cdot|^\frac{l'-1}{2}], &\qquad \text{in case II}.
\end{cases}
\end{gather*}
Thus in both cases \ref{eq:caseI} and \ref{eq:caseII} this comparison leads to the conclusion $l' = l-2k$.

We now use the other condition: $\chi_V\Delta'' \rtimes \Theta_{l'}(\tau)$ has an irreducible subquotient isomorphic to $\pi_0$. The following lemma shows that this is only possible if $l' = l_0$.

\begin{Lem}
\label{lem:main_lemma}
If $l' > l_0$ then $\chi_V\Delta'' \rtimes \Theta_{l'}(\tau)$ does not contain a subquotient isomorphic to $\pi_0$.
\end{Lem}

\begin{proof}
Recall that we have fixed the tower $\mathcal{V}$ and that $l_0$ is the smallest odd positive integer such that $\theta_{-l_0}(\pi_{00})$ is non-zero. Also recall that $\theta_{-l_0}(\pi_{00})$ is tempered, whereas all the lifts $\theta_{-l'}(\pi_{00})$ for $l' > l_0$ are non-tempered (see Theorem \ref{thm:Muic_disc}). We assume that $\pi_0$ is a subquotient of $\Theta_{l'}(\tau)$ (in particular, $\Theta_{l'}(\tau) \neq 0$) and argue by contradiction. 

Introducing the tempered support, we write $\tau$ as a direct summand of
\[
\label{eq:tempsupp}\tag{\textasteriskcentered}
\chi_W\delta_1 \times \dotsm \times \chi_W\delta_i \times \chi_W(\textnormal{St}_{l'},h) \rtimes \tau_d,
\]
where $\delta_1, \dotsc, \delta_i, \tau_d$ are discrete series representations, and $h$ denotes the number of occurrences of $\chi_W\textnormal{St}_{l'}$ in the tempered support (also, $(\textnormal{St}_{l'},h) = \textnormal{St}_{l'} \times \dotsm \times \textnormal{St}_{l'}$ $h$ times). Choosing the appropriate subquotient $\tau'$ of $\chi_W(\textnormal{St}_{l'},h) \rtimes \tau_d$ and using Corollary \ref{cor:theta_epi}, we obtain
\[
\chi_V\delta_1 \times \dotsm \times \chi_V\delta_i \rtimes \Theta_{l'}(\tau') \twoheadrightarrow  \Theta_{l'}(\tau).
\]
This allows us to view $\pi_0$ as a subquotient of $\chi_V\Delta_1 \rtimes \Theta_{l'}(\tau')$, where $\Delta_1$ is used to denote $\Delta'' \times \delta_1 \times \dotsm \times \delta_i$. We now consider the multiplicity of $\chi_WS_{l'}$ in the parameter of $\tau$, which we denote by $m_{\phi_\tau}(\chi_WS_{l'})$. Note that $\Theta_{l'}(\tau) \neq 0$ implies $m_{\phi_\tau}(\chi_WS_{l'}) > 0$. We consider the following cases separately:
\begin{itemize}
\item[1.] $m_{\phi_\tau}(\chi_WS_{l'})$ is odd, i.e.~$m_{\phi_\tau}(\chi_WS_{l'}) = 2h+1 > 0$;
\item[2.] $m_{\phi_\tau}(\chi_WS_{l'})$ is even, i.e.~$m_{\phi_\tau}(\chi_WS_{l'}) = 2h > 0$.
\end{itemize}

\bigskip

\noindent Case 1. We can use Proposition \ref{prop:Muic_isotype} and the same arguments as in the proof of Corollary \ref{cor:shaving}. We get that any subquotient of $\Theta_{l'}(\tau')$ is a subquotient of one of the following:
\begin{itemize}
\item $\chi_V(\textnormal{St}_{l'},h)\rtimes \Theta_{l'}(\tau_d)$;
\item $\chi_V(\textnormal{St}_{l'},h-1)\times \chi_V\langle|\cdot|^\frac{l'-3}{2},|\cdot|^\frac{1-l'}{2}\rangle \rtimes \Theta_{l'-2}(\tau_d)$.
\end{itemize}
Both $\Theta_{l'}(\tau_d)$ and $\Theta_{l'-2}(\tau_d)$ are irreducible discrete series representations by Theorem \ref{thm:Muic_disc}. Furthermore, by Proposition III.4.1 (ii) of \cite{kudla1996notes}, $\Theta_{l'-2}(\tau_d)$ is a subquotient of $\chi_V|\cdot|^{\frac{l'-1}{2}} \rtimes \Theta_{l'}(\tau_d)$, so $\pi_0$ is in fact a subquotient of 
\begin{align*}
&\chi_V\Delta_1 \times \chi_V(\textnormal{St}_{l'},h)\rtimes \Theta_{l'}(\tau_d) \label{eq:option1}\tag{i}\\
 \text{or} \quad &\chi_V\Delta_1 \times \chi_V(\textnormal{St}_{l'},h-1)\times \chi_V\langle|\cdot|^\frac{l'-3}{2},|\cdot|^\frac{1-l'}{2}\rangle \times \chi_V|\cdot|^{\frac{l'-1}{2}} \rtimes \Theta_{l'}(\tau_d) \label{eq:option2}\tag{ii}.
\end{align*}
Consider the representation $\chi_V\langle|\cdot|^\frac{l'-3}{2},|\cdot|^\frac{1-l'}{2}\rangle \times \chi_V|\cdot|^{\frac{l'-1}{2}}$ which appears in \eqref{eq:option2}. It has two irreducible subquotients, namely $\chi_V\textnormal{St}_{l'}$ and the corresponding Langlands subrepresentation which we denote by $L$ (see \S\ref{sec:interlude}). Therefore, any irreducible subquotient of \eqref{eq:option2} is either a subquotient of \eqref{eq:option1}, or a subquotient of $\chi_V\Delta \times \chi_V(\textnormal{St}_{l'},h-1)\times L \rtimes \Theta_{l'}(\tau_d)$. The latter representation, however, cannot contain $\pi_0$: Since (ii) contains only one irreducible $\chi$-generic subquotient (by multiplicity one), this claim follows from the fact that $L$ is not $\chi$-generic, combined with the exactness of the twisted Jacquet functor (i.e.~additivity of genericity). This shows that $\pi_0$ is necessarily a subquotient of \eqref{eq:option1}.

By the uniqueness of the tempered support (Lemma \ref{lemma:temp_supp}), we now conclude that $\Theta_{l'}(\tau_d) = \pi_{00}$. However, this implies that $\theta_{-l'}(\pi_{00}) = \tau_d$ is in discrete series despite $l' > l_0$. This contradicts the remarks at the beginning of this proof and shows that $l' > l_0$ is impossible.

\bigskip
\noindent Case 2. In this case, $\Theta_{l'}(\tau_d) = 0$. Recall that $\tau'$ is a direct summand of $\chi_W(\textnormal{St}_{l'},h) \rtimes \tau_d$. Choosing the appropriate irreducible (tempered) subquotient $\tau_1$ of $\chi_W\textnormal{St}_{l'} \rtimes \tau_d$, we define $\tau_{j+1} =\chi_W\textnormal{St}_{l'} \rtimes \tau_{j}$ for $j = 1, 2, \dotsc, h-1$ so that $\tau_h = \tau'$. Thus, $\pi_0$ is a subquotient of 
\[
\chi_V\Delta_1 \rtimes \Theta_{l'}(\tau_h).
\]
Just like in Case 1, this means that $\pi_0$ is a subquotient of
\begin{align*}
&\chi_V\Delta_1 \times \chi_V\textnormal{St}_{l'}\rtimes \Theta_{l'}(\tau_{h-1}) \label{eq:option11}\tag{i}\\
 \text{or} \quad &\chi_V\Delta_1 \times \chi_V\langle|\cdot|^\frac{l'-3}{2},|\cdot|^\frac{1-l'}{2}\rangle \times \chi_V|\cdot|^{\frac{l'-1}{2}} \rtimes \theta_{l'}(\tau_{h-1}) \label{eq:option22}\tag{ii}
\end{align*}
Note that in (ii), we have used the fact that $\tau_{h-1}$ does not contain $\chi_W\textnormal{St}_{l'-2}$ in its tempered support, so that $\Theta_{l'-2}(\tau_{h-1}) = \theta_{l'-2}(\tau_{h-1})$ is irreducible. Again, the same heredity argument we used in Case 1 now shows that $\pi_0$ must be a subquotient of \eqref{eq:option11}.

Repeating this argument $h-1$ times, we get that $\pi_0$ is a subquotient of
\[
\chi_V\Delta_1 \times \chi_V(\textnormal{St}_{l'},h-1) \rtimes \Theta_{l'}(\tau_{1}).
\]
We now claim that $\Theta_{l'}(\tau_{1})$ is irreducible. Indeed, Proposition \ref{prop:Lpackets} shows that all the irreducible subquotients of $\Theta_{l'}(\tau_{1})$ are tempered and belong to the same $L$-packet. Furthermore, by Theorem 4.3 of \cite{atobe2017local}, the parameter of this $L$-packet is multiplicity-free (it contains $\chi_VS_{l'}$ with multiplicity $1$), so that all the representations in this $L$-packet are in discrete series. This implies that $\Theta_{l'}(\tau_{1})$ is semisimple, and therefore (by Howe duality), an irreducible discrete series representation.

The fact that $\Theta_{l'}(\tau_{1})$ is irreducible shows that $\chi_V\Delta_1 \times \chi_V(\textnormal{St}_{l'},h-1) \rtimes \Theta_{l'}(\tau_{1})$ is the tempered support of $\pi_0$. In particular, we have $\theta_{l'}(\tau_{1}) = \pi_{00}$. This forces $\theta_{-l'}(\pi_{00}) = \tau_1$ to be tempered, which is again impossible for $l' > l_0$.
\end{proof}

\noindent This completes our proof of \ref{prop:main}: we have shown that $l' = l_0$, which implies that the subquotient which participates in \eqref{eq:epi1} is equal to $\theta_{-l}(\pi_{00})$. Therefore, we have
\[
\label{eq:epi2}\tag{2}
\chi_W\delta_r\nu^{s_r} \times \dotsm \times \chi_W\delta_1\nu^{s_1} \times \chi_W\Delta \rtimes \theta_{-l}(\pi_{00}) \twoheadrightarrow \sigma.
\]
\end{proof}

\subsection{Determining the standard modules}
\label{subs:standard_modules}
The above epimorphism \eqref{eq:epi2} provides plenty of information, but is not sufficient to uniquely determine $\sigma$. To do this, we have to find the standard module of $\sigma$; we do so in this section. Before we start, let us return for a moment to \eqref{eq:epi0}, Section \ref{subs:subq}. Our goal is to show two things (see Theorem \ref{thm:lifts}):
\begin{itemize}
\item the subquotient of $\Theta_{-l}(\pi_0)$ which participates in that epimorphism is $\theta_{-l}(\pi_0)$;
\item the standard module of $\sigma$ is obtained by adding $\chi_W\delta_1\nu^{s_1}, \dotsc, \chi_W\delta_r\nu^{s_r}$ to the standard module of $\theta_{-l}(\pi_0)$ (and sorting the representations decreasingly with respect to the exponents).
\end{itemize}
The shape of $\theta_{-l}(\pi_0)$ is completely determined by Theorems 4.3 and 4.5 of \cite{atobe2017local}; as it is useful to have in mind during the ensuing calculations, we compile the results of these theorems in the following proposition. In this section we fix $e = \frac{l-1}{2}$.

\begin{Prop}
\label{prop:lifts_temp}
Let $\pi_0 \in \Irr(G(W_n))$ be tempered and $\chi$-generic and let $l > 1$ be odd. We denote by $(\phi,\eta)$ the L-parameter of $\pi_0$, with $\eta$ fixed by the choice of Whittaker datum as explained in \S\ref{subs:LLC}.
\begin{enumerate}[(i)]
\item If $m_{\phi}(\chi_V)$ is even, we have two cases:\\
If $m_{\phi}(\chi_V) = 0$, then $\theta_{-1}(\pi_0)$ is the first appearance of $\pi_0$ on both towers. For $l > 1$ we have the following standard module:
	\[
	\chi_W|\cdot|^e \times \dotsm \times \chi_W|\cdot|^1 \rtimes \theta_{-1}(\pi_0) \twoheadrightarrow \theta_{-l}(\pi_0).
	\]
If $\chi_V$ appears in the tempered support $h > 0$ times, then on the going-down tower, we have
		\[
	\chi_W|\cdot|^e \times \dotsm \times \chi_W|\cdot|^1 \rtimes \theta_{-1}(\pi_0) \twoheadrightarrow \theta_{-l}(\pi_0).
	\]
	while the following holds on the going-up tower:
		\[
	\chi_W|\cdot|^e \times \dotsm \times \chi_W|\cdot|^2 \times \chi_W\textnormal{St}_2\nu^{\frac{1}{2}} \times (\chi_W,h-1) \rtimes \theta_{-1}(\pi_0') \twoheadrightarrow \theta_{-l}(\pi_0).
	\]
	Here $\pi_0'$ is the (unique) irreducible tempered representation such that $\pi_0 \hookrightarrow (\chi_V,h) \rtimes \pi_0'$.

\item If $m_{\phi}(\chi_V)$ is odd and $\mathcal{V}$ is the going-down tower, then $\theta_{-1}(\pi_0)$ is non-zero and tempered and we have
	\[
	\chi_W|\cdot|^e \times \dotsm \times \chi_W|\cdot|^1 \rtimes \theta_{-1}(\pi_0) \twoheadrightarrow \theta_{-l}(\pi_0).
	\]
\item If $m_{\phi}(\chi_V)$ is odd and $\mathcal{V}$ is the going-up tower, then $\theta_{-3}(\pi_0)$ is non-zero and tempered and we have
		\[
	\chi_W|\cdot|^e \times \dotsm \times \chi_W|\cdot|^2 \rtimes \theta_{-3}(\pi_0) \twoheadrightarrow \theta_{-l}(\pi_0).
	\]
\end{enumerate}
\end{Prop}

Our proof starts by analyzing the map \eqref{eq:epi2} established by Proposition \ref{prop:main}. We have the following cases (each of them corresponding to one of the cases of the previous proposition):
\begin{enumerate}[1.]
\item $m_\phi(\chi_V)$ is even (both towers).
\begin{itemize}
\item[1a.] $m_\phi(\chi_V) = 0$.
\item[1b.] $m_\phi(\chi_V) = 2h > 0$.
\end{itemize}
\item $m_\phi(\chi_V)$ is odd, going-down tower.
\item $m_\phi(\chi_V)$ is odd, going-up tower.
\begin{itemize}
\item[3a.] $m_\phi(\chi_VS_3)$ is odd.
\item[3b.] $m_\phi(\chi_VS_3)$ is even.
\end{itemize}
\end{enumerate}
All the cases share the same basic approach and result in analogous conclusions. However, we do have to treat them separately, mainly because of the exceptional cases which arise in some of them. Case 1a contains all the key ideas (and no tricky exceptions), so we present it in full detail.

\noindent\underline{\textbf{Case 1a:}{ $m_\phi(\chi_V) = 0$}}\\
In this case we know that on both towers, $\theta_{-l}(\pi_{00})$ is the Langlands quotient of
	\[
	\chi_W|\cdot|^e \times \dotsm \times \chi_W|\cdot|^1 \rtimes \theta_{-1}(\pi_{00}).
	\]
This also implies that $\theta_{-l}(\pi_{00})$ is the unique quotient of
\[
\chi_W\langle|\cdot|^1,|\cdot|^e\rangle \rtimes \theta_{-1}(\pi_{00})
\]
(see the notation of Section \ref{sec:interlude}). Combining this with the epimorphism in \eqref{eq:epi2} we get
\[
\chi_W\delta_r\nu^{s_r} \times \dotsm \times \chi_W\delta_1\nu^{s_1} \times \chi_W\Delta \times \chi_W\langle|\cdot|^1,|\cdot|^e\rangle \rtimes \theta_{-1}(\pi_{00}) \twoheadrightarrow \sigma.
\]
We now use Lemma \ref{lemma:irr_zel}: $\chi_W\langle|\cdot|^1,|\cdot|^e\rangle$ can switch places with all the $\delta_i'$ appearing in $\Delta$, because none of them are equal to the trivial character. This means that we can write
\[
\label{eq:epi3}\tag{3}
\chi_W\delta_r\nu^{s_r} \times \dotsm \times \chi_W\delta_1\nu^{s_1} \times \chi_W\langle|\cdot|^1,|\cdot|^e\rangle \times \chi_W\Delta \rtimes \theta_{-1}(\pi_{00}) \twoheadrightarrow \sigma.
\]
Finally, let $\tau$ be an irreducible subquotient of $\chi_W\Delta \rtimes \theta_{-1}(\pi_{00})$ such that
\[
\label{eq:epi4}\tag{4}
\chi_W\delta_r\nu^{s_r} \times \dotsm \times \chi_W\delta_1\nu^{s_1} \times \chi_W\langle|\cdot|^1,|\cdot|^e\rangle \rtimes \tau \twoheadrightarrow \sigma.
\]
Note that $\tau$ is tempered, because $\theta_{-1}(\pi_{00})$ is, too (moreover, in this case, $\theta_{-1}(\pi_{00})$ is in discrete series), as are all the irreducible subquotients of $\Delta$.

We now claim that the representation appearing on the left-hand side of \eqref{eq:epi4} has a unique irreducible quotient. A similar statement appears in all the other cases, so we prove this statement in somewhat greater generality.
\begin{Lem}
\label{lemma:unique_quotient}
Let $a_i,b_i \in \mathbb{R}$ be such that $b_i-a_i \in \mathbb{Z}_{\geq 0}$ and let $\rho_i$ be an irreducible unitary cuspidal representation of $\GL_{d_i}(\F)$ for $i = 1,\dotsc, r$. Set $\delta_i\nu^{s_i} = \langle |\cdot|^{b_i}\rho_i, |\cdot|^{a_i}\rho_i \rangle$ with $s_i = \frac{a_i+b_i}{2}$, $i=1,\dotsc,r$ and assume that $s_r \geq \dotsb \geq s_1>0$. Furthermore, let $\tau$ be an irreducible tempered representation of $H(V_m)$.

Consider the representation
\[
\label{eq_lemma7.5}\tag{\textasteriskcentered}
\chi_W\delta_r\nu^{s_r} \times \dotsm \times \chi_W\delta_j\nu^{s_j} \times \chi_W\langle|\cdot|^{c_0},|\cdot|^{d_0}\rangle \times \chi_W\delta_{j-1}\nu^{s_{j-1}} \times \dotsm \times \chi_W\delta_1\nu^{s_1} \rtimes \tau
\]
where $c_0>0$, $d_0 \in \mathbb{R}$ such that $d_0-c_0 \in \mathbb{Z}_{\geq 0}$. Assume that $s_i < c_0$ for $i = 1,\dotsc, j-1$ and $b_i \neq c_0-1$ for $i = j, j+1,\dotsc,r$. Then the representation \eqref{eq_lemma7.5} has a unique irreducible quotient, given by
\[
L(\chi_W\delta_r\nu^{s_r},\dotsc,\chi_W\delta_1\nu^{s_1}, \chi_W|\cdot|^{d_0},\dotsc,|\cdot|^{c_0}; \tau)
\]
\end{Lem}

\begin{proof}
We will show that the representation \eqref{eq_lemma7.5} is itself a quotient of a standard module, and the conclusion will follow. This standard module is obtained by inserting the representations $\chi_W|\cdot|^k$ (for $k = c_0,\dotsc, d_0$) among the representations $\delta_i\nu^{s_i}$ in the standard module
\[
\chi_W\delta_r\nu^{s_r} \times \dotsm \times \chi_W\delta_1\nu^{s_1} \rtimes \tau
\]
so that the ordering of the exponents is preserved. 

If $s_j \geqslant d_0$ (or if there are no representations $ \chi_W\delta_j\nu^{s_j},\dotsc, \chi_W\delta_r\nu^{s_r}$), then the representation \eqref{eq_lemma7.5} is a quotient of the standard module
\[
\chi_W\delta_r\nu^{s_r} \times \dotsm \times \chi_W\delta_j\nu^{s_j} \times \chi_W|\cdot|^{d_0}\times \dotsm \times \chi_W|\cdot|^{c_0} \times \chi_W\delta_{j-1}\nu^{s_{j-1}} \times \dotsm \times \chi_W\delta_1\nu^{s_1} \rtimes \tau
\]
and we are done. If $s_j < d_0$ we use the following technical observation, based on Lemma \ref{lemma:irr_zel}.
\begin{Lem}
\label{lem:tricky_condition}
Let $\sigma$ be an irreducible representation of $H(V_m)$ and let $a,b,c,d \in \mathbb{R}$ be numbers such that $a\leq b$, $c\leq d$ and $\frac{a+b}{2} \leqslant d$. Assume that
\[
A \times \langle|\cdot|^b,|\cdot|^a\rangle \times \langle|\cdot|^c,|\cdot|^d\rangle \rtimes \sigma_0 \twoheadrightarrow \sigma
\]
for some representations $A$ and $\sigma_0$ . If
\[
\label{eq:tricky1}\tag{i}
c \neq b+1
\]
then setting $s =\mathrm{min}\{s''\in[c,d]:s''\ge \frac{a+b}{2}\}$ we have
\[
A \times \langle|\cdot|^s,|\cdot|^d\rangle \times \langle|\cdot|^b,|\cdot|^a\rangle \times  \langle|\cdot|^c,|\cdot|^{s-1}\rangle \rtimes \sigma_0 \twoheadrightarrow \sigma.
\]
Here, the segment $[c,s-1]$ can be empty (in that case, we omit $\langle|\cdot|^c,|\cdot|^{s-1}\rangle$). Assume, a fortiori, that
\[
\label{eq:tricky2}\tag{ii}
c-1 \lneqq \frac{a+b}{2};
\]
notice that this implies $c\neq b+1$. Then $s-1 \lneqq \frac{a'+b'}{2}$ for any $a',b'$ such that $\frac{a'+b'}{2} \geq \frac{a+b}{2}$.
\end{Lem}

\begin{proof}
We know that $ \langle|\cdot|^c,|\cdot|^d\rangle$ is a quotient of $ \langle|\cdot|^s,|\cdot|^d\rangle \times  \langle|\cdot|^c,|\cdot|^{s-1}\rangle$, so we have
\[
A \times \langle|\cdot|^b,|\cdot|^a\rangle \times \langle|\cdot|^s,|\cdot|^d\rangle \times \langle|\cdot|^c,|\cdot|^{s-1}\rangle \rtimes \sigma_0 \twoheadrightarrow \sigma.
\]
If \eqref{eq:tricky1} holds, then Lemma \ref{lemma:irr_zel} shows that $\langle|\cdot|^b,|\cdot|^a\rangle$ and $\langle|\cdot|^s,|\cdot|^d\rangle$ can switch places. We thus get
\[
A \times \langle|\cdot|^s,|\cdot|^d\rangle \times \langle|\cdot|^b,|\cdot|^a\rangle \times \langle|\cdot|^c,|\cdot|^{s-1}\rangle \rtimes \sigma_0 \twoheadrightarrow \sigma,
\]
as required.

For the second part of the claim, assume condition \eqref{eq:tricky2} is satisfied. Then (i) obviously holds because $a \leq b$. Thus $s -1 \lneqq \frac{a'+b'}{2}$ is also obvious since we have $s-1 \lneqq \frac{a+b}{2} \leq \frac{a'+b'}{2}$.
\end{proof}

\noindent Lemma \ref{lem:tricky_condition} allows us to prove Lemma \ref{lemma:unique_quotient} inductively with respect to $(d_0-c_0)+(r-j)$ (the length of $[c_0,d_0]$ plus the number of representations appearing to the left of $\chi_W\langle|\cdot|^{c_0},|\cdot|^{d_0}\rangle$). The case when $r-j = -1$ (i.e.~no representations to the left) is treated before Lemma \ref{lem:tricky_condition}, and so is the case $s_j \geq d_0$. We therefore assume $s_j < d_0$. The base case $d_0-c_0=0$ is also trivial: since no $b_i$ is equal to $c_0-1$, the segment $[|\cdot|^{c_0},|\cdot|^{c_0}]$ is not linked to any of the segments which define $\delta_i\nu^{s_i}$ for $i\geq j$. Therefore $|\cdot|^{c_0}$ can switch places with any $\delta_i\nu^{s_i}$ (with $i\geq j$) such that $s_i < c_0$. We can thus rearrange \eqref{eq_lemma7.5} to obtain the standard module.

To perform the induction step, assume that $d_0-c_0 > 0$. First, if $c_0 - 1 \geq \frac{a_j+b_j}{2}$, we may use Lemma \ref{lemma:irr_zel}: the segments $[|\cdot|^{c_0},|\cdot|^{d_0}]$ and $[|\cdot|^{a_j},|\cdot|^{b_j}]$ are not adjacent because of our assumption that $b_j \neq c_0-1$. Therefore, we may switch the places of $\delta_j\nu^{s_j}$ and $\chi_W\langle|\cdot|^{c_0},|\cdot|^{d_0}\rangle$.

Performing this switch a suitable number of times, we may assume that $c_0 - 1 < \frac{a_j+b_j}{2}$ so that condition (ii) of Lemma \ref{lem:tricky_condition} is fulfilled. We may use thus Lemma \ref{lem:tricky_condition} (with $\langle|\cdot|^b,|\cdot|^a\rangle = \delta_j\nu^{s_j}$ and $\langle|\cdot|^c,|\cdot|^d\rangle = \langle|\cdot|^{c_0},|\cdot|^{d_0}\rangle$) to show that we can write $\langle|\cdot|^s,|\cdot|^{d_0}\rangle \times\delta_j\nu^{s_j} \times  \langle|\cdot|^{c_0},|\cdot|^{s-1}\rangle$ instead of $\delta_j\nu^{s_j} \times \langle|\cdot|^{c_0},|\cdot|^{d_0}\rangle$ in \eqref{eq_lemma7.5}.
(Notice that Lemma \ref{lem:tricky_condition} is not needed here when $\rho_j \neq \mathbb{1}$). Writing $|\cdot|^{s-1} \times \dotsm\times|\cdot|^{c_0}$ instead of $\langle|\cdot|^{c_0},|\cdot|^{s-1}\rangle$, we now arrive at
\begin{align*}
\chi_W\delta_r\nu^{s_r} \times \dotsm \times \chi_W\delta_{j+1}\nu^{s_{j+1}} &\times \langle|\cdot|^s,|\cdot|^{d_0}\rangle \times\delta_j\nu^{s_j} \times  |\cdot|^{s-1} \times \dotsm\times|\cdot|^{c_0}\\
& \times \chi_W\delta_{j-1}\nu^{s_{j-1}} \times \dotsm \times \chi_W\delta_1\nu^{s_1} \rtimes \tau
\end{align*}
This is certainly a representation of the form described by Lemma \ref{lemma:unique_quotient}: since we have used condition (ii), Lemma \ref{lem:tricky_condition} ensures that $b_i \neq s-1$ for $i = j+1,\dotsc,r$.
If $s > c_0$, then $d_0-s < d_0-c_0$ and the inductive step is thus completed. Otherwise, the length of the segment stays the same, but we have decreased the number of representations appearing to the left of it. Therefore, the induction step is complete. This concludes the proof.
\end{proof}

We return to our discussion preceding Lemma \ref{lemma:unique_quotient}. Applying the lemma to \eqref{eq:epi4} shows that $\chi_W\delta_r\nu^{s_r} \times \dotsm \times \chi_W\delta_1\nu^{s_1} \times \chi_W\langle|\cdot|^1,|\cdot|^e\rangle \rtimes \tau$ has a unique irreducible quotient. Note that the proof of the lemma also shows that the standard module of $\sigma$ is obtained by inserting $|\cdot|^1,\dotsc,|\cdot|^e$ among the $\delta_i\nu^{s_i}$ so that the ordering of the exponents is preserved. It remains to determine $\tau$, i.e.~the tempered part of the standard module of $\sigma$. By the uniqueness of the irreducible quotient, from \eqref{eq:epi4} we get
\[
\label{eq:epi5}\tag{5}
\chi_W\delta_r\nu^{s_r} \times \dotsm \times \chi_W\delta_1\nu^{s_1} \rtimes \tau' \twoheadrightarrow \sigma,
\]
where $\tau'$ is the unique irreducible quotient of $\chi_W\langle|\cdot|^1,|\cdot|^e\rangle \rtimes \tau$ (that is, the Langlands quotient of $\chi_W|\cdot|^e \times \dotsm \times \chi_W|\cdot|^1 \rtimes \tau$). It is now important to notice that we may assume that $\tau'$ is a subquotient of $\Theta_{-l}(\pi_0)$. This observation will be used in all the other cases as well, so we prove a more general statement.
\begin{Lem}
\label{lemma:really_subq}
Assume that the irreducible subquotient $\tau'$ of $\chi_W\Delta \rtimes \theta_{-l}(\pi_{00})$ which participates in \eqref{eq:epi2} is uniquely determined. Then there is a subquotient of $\Theta_{-l}(\pi_0)$ isomorphic to $\tau'$ which participates in \eqref{eq:epi0}.
\end{Lem}
\begin{proof}
We explain the maps \eqref{eq:epi0}--\eqref{eq:epi2} in more detail and recapitulate the preceding discussion. Let $\Pi$ denote $\delta_r\nu^{s_r} \times \dotsm \times \delta_1\nu^{s_1}$. Starting from
\[
f \colon \chi_W\Delta \rtimes \Theta_{-l}(\pi_{00}) \twoheadrightarrow \Theta_{-l}(\pi_0)
\]
we induce to obtain
\[
\text{Ind}(f)\colon \chi_W\Pi \times \chi_W\Delta \rtimes \Theta_{-l}(\pi_{00}) \twoheadrightarrow \chi_W\Pi \rtimes \Theta_{-l}(\pi_0).
\]
Now let $g \colon \chi_W\Pi \rtimes \Theta_{-l}(\pi_0) \twoheadrightarrow \sigma$ denote the epimorphism in \eqref{eq:epi0}. Composing $g$ with $\text{Ind}(f)$ we arrive at \eqref{eq:epi1}. Thus epimorphism \eqref{eq:epi1} factors through $g$:
\[
\chi_W\Pi \times \chi_W\Delta \rtimes \Theta_{-l}(\pi_{00}) \twoheadrightarrow \chi_W\Pi \rtimes {\Theta_{-l}(\pi_0)} \overset{g}{\twoheadrightarrow} \sigma. 
\]
Proposition \ref{prop:main} shows that no subquotient of $\Theta_{-l}(\pi_{00})$ except $\theta_{-l}(\pi_{00})$ can participate in the above epimorphism $g\circ \text{Ind}(f)$. In other words, we have $\chi_W\Pi \times \chi_W\Delta \rtimes \Theta^0 \subseteq \ker g\circ \text{Ind}(f)$ where we have used $\Theta^0$ to denote the maximal proper subrepresentation of $\Theta_{-l}(\pi_{00})$.

We may therefore take the quotient of $g \circ \text{Ind}(f)$ by $\chi_W\Pi \times \chi_W\Delta \rtimes \Theta^0$ to obtain the map $\widetilde{g \circ \text{Ind}(f)}$, which gives us the epimorphism in \eqref{eq:epi2}. Thus \eqref{eq:epi2} may be expressed as
\[
\chi_W\Pi \times \chi_W\Delta \rtimes \theta_{-l}(\pi_{00}) \cong \chi_W\Pi \times \chi_W\Delta \rtimes (\Theta_{-l}(\pi_{00})/\Theta^0)  \twoheadrightarrow \chi_W\Pi \rtimes (\Theta_{-l}(\pi_0)/N) \overset{\tilde{g}}{\twoheadrightarrow} \sigma,
\]
where $N$ denotes the subrepresentation of $\Theta_{-l}(\pi_0)$ equal to $f(\chi_W\Delta \rtimes \Theta^0)$, and $\tilde{g}$ denotes the corresponding quotient of the map $g$.

We now take any irreducible subquotient $\tau''$ of $\Theta_{-l}(\pi_0)/N$ which participates in the map $\tilde{g}$. Since $\Theta_{-l}(\pi_0)/N$ is a quotient of $\Theta_{-l}(\pi_{0})$, $\tau''$ appears as a subquotient of $\Theta_{-l}(\pi_{0})$ which participates in $g$, i.e.~\eqref{eq:epi0}. We claim that $\tau''$ is necessarily isomorphic to $\tau'$.

Indeed, $\Theta_{-l}(\pi_0)/N$ is at the same time a quotient of $\chi_W\Delta \rtimes \theta_{-l}(\pi_{00})$, so $\tau''$ also appears as a subquotient of $\chi_W\Delta \rtimes \theta_{-l}(\pi_{00})$ which participates in the epimorphism displayed above, i.e.~\eqref{eq:epi2}. By the assumption in the statement of this lemma, such a subquotient is determined uniquely. Therefore, we have $\tau'' \cong \tau'$. This proves the lemma.
\end{proof}

The discussion preceding Lemma \ref{lemma:really_subq}---in particular, Lemma \ref{lemma:unique_quotient}---shows that in case 1a, $\tau'$ is uniquely determined: it is the Langlands quotient of $\chi_W|\cdot|^e \times \dotsm \times \chi_W|\cdot|^1 \rtimes \tau$, where $\tau$ is the tempered part in the standard representation of $\sigma$. Now Lemma \ref{lemma:really_subq} shows that $\tau'$ appears as a subquotient of $\Theta_{-l}(\pi_0)$ participating \eqref{eq:epi0}, so it remains to see the that:
\begin{Lem}
\label{lemma:really_small_theta}
The only irreducible subquotient of $\Theta_{-l}(\pi_0)$ with standard module of the form $\chi_W|\cdot|^e \times \dotsm \times \chi_W|\cdot|^1 \rtimes \tau$ is $\theta_{-l}(\pi_{0})$.
\end{Lem}

To prove this, we first show the following:
\begin{Lem}
\label{lem_descend}
Let $l>1$ be an odd integer and let $\pi_0 \in \Irr (G(W_{n_0}))$ be a tempered representation such that $\Theta_{-l}(\pi_0) \neq 0$. Let $\tau_1$ be an irreducible representation such that $\Theta_{-l}(\pi_0)$ possesses an irreducible subquotient which is at the same time a quotient of $\chi_W|\cdot|^{\frac{l-1}{2}} \rtimes \tau_1$. Then $\tau_1$ is a subquotient of $\Theta_{2-l}(\pi_0)$ (in particular, $\Theta_{2-l}(\pi_0) \neq 0$).
\end{Lem}

\begin{proof}
Set $e= \frac{l-1}{2}$ and let $\tau'$ be an irreducible subquotient of $\Theta_{-l}(\pi_0)$ such that
\[
\chi_W|\cdot|^e \rtimes \tau_1 \twoheadrightarrow \tau',
\]
i.e.~$\tau' \hookrightarrow \chi_W|\cdot|^{-e} \rtimes \tau_1$. We use Kudla's filtration; a similar computation is used in the proof of Theorem 4.1 of \cite{muic2008structure}. The map we have just obtained shows that $\Hom(\tau', \allowbreak \chi_W|\cdot|^{-e} \allowbreak \rtimes \tau_1) \neq 0$. By Frobenius reciprocity, this means that $\Hom(R_{Q_1}(\tau'), \chi_W|\cdot|^{-e} \otimes \tau_1) \neq 0$, where $Q_1$ denotes the appropriate standard maximal parabolic subgroup of $H(V_m)$. (Also, the $\Hom$-spaces above are the intertwining spaces for the actions of $H(V_m)$ and the Levi factor of $Q_1$, respectively). Since $\tau'$ is a subquotient of $\Theta_{-l}(\pi_0)$, $\chi_W|\cdot|^{-e} \otimes \tau_1$ is a subquotient of $R_{Q_1}(\Theta_{-l}(\pi_0))$. We know that $R_{Q_1}(\Theta_{-l}(\pi_0))$ is admissible (because $\Theta_{-l}(\pi_0)$ is), so we may decompose it as
\[
R_{Q_1}(\Theta_{-l}(\pi_0)) = \bigoplus_\mu R_{Q_1}(\Theta_{-l}(\pi_0))_\mu
\]
where $\mu$ runs over (a finite set of) characters of $\GL_1(\F)$, and $R_{Q_1}(\Theta_{-l}(\pi_0))_\mu$ denotes the maximal subrepresentation of $R_{Q_1}(\Theta_{-l}(\pi_0))$ on which $\GL_1(\F)$ acts by $\mu$. Therefore, $\chi_W\allowbreak|\cdot|^{-e} \otimes \tau_1$ is a subquotient of $R_{Q_1}(\Theta_{-l}(\pi_0))_{\chi_W|\cdot|^{-e}}$.

On the other hand, $\pi_0 \otimes \Theta_{-l}(\pi_0)$ is a quotient of $\omega_{m_0,n_0}$---here $n_0$ is defined by $\pi_0 \in \Irr(G(W_{n_0}))$, $m_0 = n_0 + \epsilon + l$, and $\omega_{m_0,n_0}$ is the corresponding Weil representation. By the exactness of $R_{Q_1}$, $\pi_0 \otimes R_{Q_1}(\Theta_{-l}(\pi_0))$ is a quotient of $R_{Q_1}(\omega_{m_0,n_0})$. Therefore, $\pi_0 \otimes R_{Q_1}(\Theta_{-l}(\pi_0))_{\chi_W|\cdot|^{-e}}$ is also a quotient of $R_{Q_1}(\omega_{m_0,n_0})$. Recall that $e = \frac{l-1}{2}$. Kudla's filtration of $R_{Q_1}(\omega_{m_0,n_0})$ is
\begin{align*}
J^0 &= \chi_W|\cdot|^{-e} \otimes \omega_{m_0-2,n_0}\quad (\text{the quotient})\\
J^1 &= \text{Ind}_{P_1\times \GL_1(\F) \times H_{m_0-2}}^{G_{n_0} \times \GL_1(\F) \times H_{m_0-2}}(\Sigma_1 \otimes \omega_{m_0-2,n_0-2}) \quad (\text{the subrepresentation}).
\end{align*}
We claim that $J^1$ cannot participate in the map $R_{Q_1}(\omega_{m_0,n_0})\twoheadrightarrow \pi_0 \otimes R_{Q_1}(\Theta_{-l}(\pi_0))_{\chi_W|\cdot|^{-e}}$. Assume the contrary, i.e. that $\Hom_{G_{n_0} \times \GL_1(\F) \times H_{m_0-2}}(J^1 ,\pi_0 \otimes R_{Q_1}(\Theta_{-l}(\pi_0))_{\chi_W|\cdot|^{-e}}) \neq 0$. Then an application of the second Frobenius reciprocity shows that
\[
\Hom_{\GL_1(\F) \times G_{n_0-2} \times \GL_1(\F) \times H_{m_0-2}}(\Sigma_1 \otimes \omega_{m_0-2,n_0-2},R_{\overline{P_1}}(\pi_0) \otimes R_{Q_1}(\Theta_{-l}(\pi_0))_{\chi_W|\cdot|^{-e}}) \neq 0,
\]
where $\overline{P}_1$ denotes the parabolic subgroup opposite to $P_1$.  Thus there is an irreducible subquotient, say $\chi_W|\cdot|^{-e} \otimes \pi_1$, of $R_{Q_1}(\Theta_{-l}(\pi_0))_{\chi_W|\cdot|^{-e}}$ such that
\[
\Hom_{\GL_1(\F) \times G_{n_0-2} \times \GL_1(\F) \times H_{m_0-2}}(\Sigma_1 \otimes \omega_{m_0-2,n_0-2},R_{\overline{P_1}}(\pi_0) \otimes \chi_W|\cdot|^{-e} \otimes \pi_1) \neq 0.
\]
This implies $R_{\overline{P_1}}(\pi_0)$ has a subrepresentation of the form $\chi_V|\cdot|^e \otimes \pi_1'$ for some $\pi_1' \neq 0$. As $\pi_0$ is tempered, and $e > 0$, Casselman's criterion shows that this is impossible.

We have thus shown that $\pi_0 \otimes R_{Q_1}(\Theta_{-l}(\pi_0))_{\chi_W|\cdot|^{-e}}$ is a quotient of $J^0$, which immediately implies that $\tau_1$ is a subquotient of $\Theta_{2-l}(\pi_0)$.
\end{proof}

\begin{proof}[Proof of Lemma \ref{lemma:really_small_theta}]
Let $\tau'$ be a subquotient of $\Theta_{-l}(\pi_0)$ such that
\[
\chi_W|\cdot|^e \times \dotsm \times \chi_W|\cdot|^1 \rtimes \tau \twoheadrightarrow \tau'
\]
for some tempered $\tau$. It suffices to prove that $\tau \cong \theta_{-1}(\pi_0)$; Proposition \ref{prop:lifts_temp} then shows that $\tau' \cong \theta_{-l}(\pi_0)$. We do this by applying Lemma \ref{lem_descend} inductively with respect to $l$.

For the base case, $l=3$ and we have $\chi_W|\cdot|^1\rtimes \tau \twoheadrightarrow \tau'$. Applying Lemma \ref{lem_descend} with $l=3$ and $\tau_1 = \tau$ we get that $\tau$ is a subquotient of $\Theta_{-1}(\pi_0)$. Since $\Theta_{-1}(\pi_0)$ is irreducible, we deduce $\tau = \theta_{-1}(\pi_0)$, as claimed.

For the induction hypothesis, assume that Lemma \ref{lemma:really_small_theta} holds when $l$ is replaced by $l-2$; in other words, assume that any irreducible subquotient of $\Theta_{2-l}(\pi_0)$ of the form $L(\chi_W|\cdot|^{e-1}, \dotsc, \allowbreak \chi_W|\cdot|^1; \tau)$ necessarily satisfies $\tau = \theta_{-1}(\pi_0)$.

Now let $\tau'$ be an irreducible subquotient of $\Theta_{-l}(\pi_0)$ of the form $L(\chi_W|\cdot|^{e}, \dotsc,  \chi_W|\cdot|^1;\tau)$. Denote by $\tau_1$ the Langlands quotient of $\chi_W|\cdot|^{e-1} \times \dotsm \times \chi_W|\cdot|^1 \rtimes \tau$, so that $\chi_W|\cdot|^e \rtimes \tau_1 \twoheadrightarrow \tau'$. Lemma \ref{lem_descend} then shows that $\tau_1$ is a subquotient of $\Theta_{2-l}(\pi_0)$. By the induction hypothesis, we now deduce that $\tau =  \theta_{-1}(\pi_0)$, as desired.
\end{proof}
This completes case (1a). Let us summarize: Lemma \ref{lemma:unique_quotient} applied to the left-hand side of \eqref{eq:epi4} shows that we have 
\[
\chi_W\Pi \rtimes \tau' \twoheadrightarrow \theta_{-l}(\pi)
\]
for some irreducible representation $\tau'$ of the form $L(\chi_W|\cdot|^{e}, \dotsc,  \chi_W|\cdot|^1;\tau)$. Lemma \ref{lemma:really_subq} then shows that we may assume that $\tau'$ is a subquotient of  $\Theta_{-l}(\pi_0)$, from where, using Lemma \ref{lemma:really_small_theta}, we deduce that $\tau' = \theta_{-l}(\pi_0)$. We have thus shown that
\[
\chi_W\Pi \rtimes \theta_{-l}(\pi_0) \twoheadrightarrow \theta_{-l}(\pi),
\]
and we have determined the standard module of $\theta_{-l}(\pi)$ (using the proof of Lemma 7.5):
\[
\theta_{-l}(\pi) = L(\chi_W\delta_r\nu^{s_r}, \dotsc, \chi_W\delta_1\nu^{s_1}, \chi_W|\cdot|^e, \dotsc, \chi_W|\cdot|^1; \theta_{-1}(\pi_0)).
\]

\noindent\underline{\textbf{Case 1b:}{ $m_\phi(\chi_V) = 2h > 0$}}\\
In this case, we know that on both towers $\theta_{-l}(\pi_{00})$ is the Langlands quotient of
	\[
	\chi_W|\cdot|^e \times \dotsm \times \chi_W|\cdot|^1 \rtimes \theta_{-1}(\pi_{00}).
	\]
that is, the unique quotient of
\[
\chi_W\langle|\cdot|^1,|\cdot|^e\rangle \rtimes \theta_{-1}(\pi_{00}).
\]
Thus, \eqref{eq:epi2} leads to
\[
\chi_W\delta_r\nu^{s_r} \times \dotsm \times \chi_W\delta_1\nu^{s_1} \times \chi_W\Delta \times \chi_W\langle|\cdot|^1,|\cdot|^e\rangle \rtimes \theta_{-1}(\pi_{00}) \twoheadrightarrow \sigma.
\]
In contrast with the previous case, we now know that $\pi_0$ contains $\chi_V$ in its tempered support; equivalently, the trivial character of $\GL_1(\F)$ appears among the representations which define $\Delta$. Moreover, it appears exactly $h$ times because of our assumption is that $m_\phi(\chi_V) = 2h$. This means that we can write $\chi_W\Delta = (\chi_W,h) \times \chi_W\Delta'$ for some appropriately chosen $\Delta'$ induced from discrete series representations.

Using Lemma \ref{lemma:irr_zel}, we can swap $\chi_W\langle|\cdot|^1,|\cdot|^e\rangle $ with all the representations appearing in $\chi_W\Delta'$. Instead of the above epimorphism, we thus get
\[
\chi_W\delta_r\nu^{s_r} \times \dotsm \times \chi_W\delta_1\nu^{s_1} \times (\chi_W,h) \times \chi_W\langle|\cdot|^1,|\cdot|^e\rangle \times \chi_W\Delta'  \rtimes \theta_{-1}(\pi_{00}) \twoheadrightarrow \sigma.
\]
Here is where the situation gets more complicated than in the previous case: according to Lemma \ref{rem:swap2}, $(\chi_W,h) \times \chi_W\langle|\cdot|^1,|\cdot|^e\rangle$ reduces. Furthermore, we know $\mathbb{1} \times \langle|\cdot|^1,|\cdot|^e\rangle$ has two irreducible subquotients; by Lemma \ref{rem:swap2} those are
\begin{enumerate}[(i)]
\item $\langle|\cdot|^0,|\cdot|^e\rangle$ (the unique irreducible subrepresentation of $\mathbb{1} \times \langle|\cdot|^1,|\cdot|^e\rangle$);
\item $L =$(the unique irreducible) quotient of $\langle|\cdot|^2,|\cdot|^e\rangle \times \textnormal{St}_2\nu^{\frac{1}{2}}$.
\end{enumerate}
From here, our discussion ramifies into two possible cases, depending on the subquotient of $(\chi_W,h) \times \chi_W\langle|\cdot|^1,|\cdot|^e\rangle$ which participates in the above epimorphism (note that $L$ may switch places with $\mathbb{1}$ by Lemma \ref{lemma_specific_reducibility}):
\begin{enumerate}[(i)]
\item $\chi_W\delta_r\nu^{s_r} \times \dotsm \times \chi_W\delta_1\nu^{s_1} \times  \chi_W\langle|\cdot|^1,|\cdot|^e\rangle \times (\chi_W,h) \times \chi_W\Delta'  \rtimes \theta_{-1}(\pi_{00}) \twoheadrightarrow \sigma$;
\item $\chi_W\delta_r\nu^{s_r} \times \dotsm \times \chi_W\delta_1\nu^{s_1} \times   \chi_WL \times (\chi_W,h-1) \times \chi_W\Delta'  \rtimes \theta_{-1}(\pi_{00}) \twoheadrightarrow \sigma$.
\end{enumerate}
In both cases, we can specify the irreducible subquotient of the tempered part:
\begin{enumerate}[(i)]
\item There is an irreducible (and tempered) subquotient $\tau_1$ of 
$(\chi_W,h) \times \chi_W\Delta'  \rtimes \theta_{-1}(\pi_{00})$ such that
\[
\chi_W\delta_r\nu^{s_r} \times \dotsm \times \chi_W\delta_1\nu^{s_1} \times  \chi_W\langle|\cdot|^1,|\cdot|^e\rangle \rtimes \tau_1 \twoheadrightarrow \sigma;
\]
\item There is an irreducible (and tempered) subquotient $\tau_2$ of 
$(\chi_W,h-1) \times \chi_W\Delta'  \rtimes \theta_{-1}(\pi_{00})$ such that
\[
\chi_W\delta_r\nu^{s_r} \times \dotsm \times \chi_W\delta_1\nu^{s_1} \times   \chi_WL \times \tau_2 \twoheadrightarrow \sigma.
\]
\end{enumerate}

We now use Lemma \ref{lemma:unique_quotient} to show that the the induced representations in (i) and (ii) have unique irreducible quotients. If (i) is true, then we have a situation just like in case 1a: the left-hand side in (i) has a unique irreducible quotient, and therefore the subquotient of $\Theta_{-l}(\pi_0)$ which participates in the epimorphism is equal to the Langlands quotient of
\[
\chi_W|\cdot|^e \times \dotsm \times \chi_W|\cdot|^1 \rtimes \tau_1.
\]
If (ii) holds, we can do the same thing, except when one of the $\delta_i\nu^{s_i}$ is defined by a segment $[|\cdot|^a,|\cdot|^b]$ whose upper end $b$ is equal to $1$ (see Lemma \ref{lemma:unique_quotient}). Since $a+b > 0$, this exceptional situation occurs only if $[a,b] = [0,1]$ or $[a,b] = [1]$, i.e.~$\delta_i\nu^{s_i} = \textnormal{St}_2\nu^{\frac{1}{2}}$ or $\delta_i\nu^{s_i} = |\cdot|^1$.

We can eliminate the possibility $\delta_i\nu^{s_i} = |\cdot|^{1}$ immediately: since $\chi_V$ is contained in the tempered support of $\pi_0$, and $\chi_V|\cdot|^1 \times \chi_V$ reduces, using Proposition \ref{prop:std_mod_reduc} and \cite[Lemma 2.1 and 2.2]{muic2005reducibility} we get that the standard module of $\pi$ reduces if it contains $\chi_V|\cdot|^1$. As $\pi$ is $\chi$-generic, the standard module conjecture (Theorem \ref{thm:gen_properties} ii) shows that this is not possible.

The possibility that $\delta_
i\nu^{s_i} = \textnormal{St}_2\nu^{\frac{1}{2}}$ cannot be eliminated. On the other hand, Proposition \ref{prop:std_mod_reduc} shows that $\chi_V\textnormal{St}_2\nu^{\frac{1}{2}}$ can appear in the standard module of $\pi$ at most once; since $s_i = \frac{1}{2}$, we may assume (without loss of generality) that $i=1$. Thus, let us assume that option (ii) holds and that $\delta_1\nu^{s_1} = \textnormal{St}_2\nu^{\frac{1}{2}}$. We then have
\[
\chi_W\delta_r\nu^{s_r} \times \dotsm \times \chi_W\delta_2\nu^{s_2}  \times \chi_W\textnormal{St}_2\nu^{\frac{1}{2}} \times \chi_WL \rtimes \tau_2 \twoheadrightarrow \sigma.
\]
Lemma \ref{lemma_specific_reducibility} now shows that we can write $\chi_WL \times \chi_W\textnormal{St}_2\nu^{\frac{1}{2}}$ instead of $\chi_W\textnormal{St}_2\nu^{\frac{1}{2}} \times \chi_WL$:
\[
\chi_W\delta_r\nu^{s_r} \times \dotsm \times \chi_W\delta_2\nu^{s_2} \times  \chi_W\langle|\cdot|^2,|\cdot|^e\rangle \times \chi_W\textnormal{St}_2\nu^{\frac{1}{2}} \times \chi_W\textnormal{St}_2\nu^{\frac{1}{2}} \rtimes \tau_2 \twoheadrightarrow \sigma.
\]
We can apply Lemma \ref{lemma:unique_quotient}: we have shown that  the remaining $\delta_2\nu^{s_2}, \dotsc$, $\delta_r\nu^{s_r}$ are defined by segments which do not end in $|\cdot|^1$. Thus, even in this exceptional case, the representation on the left-hand side of (ii) has a unique irreducible quotient.

In particular, this shows that cases (i) and (ii) are mutually exclusive, because they imply different standard modules for $\sigma$. Therefore, regardless of whether (i) or (ii) is valid, we may now apply Lemma \ref{lemma:really_subq}. This way, (i) and (ii) lead to the following conclusion: the irreducible subquotient of $\Theta_{-l}(\pi_0)$ which participates in \eqref{eq:epi0} has a standard module of the form:
\begin{enumerate}[(i)]
\item $\chi_W|\cdot|^e \times \dotsm \times \chi_W|\cdot|^1 \rtimes \tau_1$; or
\item $\chi_W|\cdot|^e \times \dotsm \times \chi_W|\cdot|^2 \times \chi_W\textnormal{St}_2\nu^{\frac{1}{2}} \rtimes \tau_2$.
\end{enumerate}
It remains to see whether (i) or (ii) is valid, and to determine $\tau_1$ (resp.~$\tau_2$). To this end, we first use Lemma \ref{lem_descend} inductively, just like in the proof of Lemma \ref{lemma:really_small_theta}. It shows the following: 
\begin{itemize}
\item if (i) holds, then the (unique) irreducible quotient of $\chi_W|\cdot|^1 \rtimes \tau_1$ is isomorphic to a subquotient of $\Theta_{-3}(\pi_0)$;
\item if (ii) holds, then the (unique) irreducible quotient of $\chi_W\textnormal{St}_2\nu^{\frac{1}{2}} \rtimes \tau_2$ is isomorphic to a subquotient of $\Theta_{-3}(\pi_0)$.
\end{itemize}
We will show that (i) holds for the lifts to the going-down tower, whereas (ii) holds for the lifts to the going-up tower. To prove this, we need another auxiliary result; we state and prove it in a general form which will allow us to also apply it to other cases we are considering. Recall that $m_\phi(\chi_V) = 2h>0$, i.e.~that $\chi_V$ appears $h$ times in the tempered support of $\pi_0$. This means that there is an irreducible representation $\pi_0'$ (whose parameter does not contain $\chi_V$) such that
\[
\pi_0 \hookrightarrow (\chi_V,h) \rtimes \pi_0'.
\]
The right-hand side in the above equation is completely reducible; since $\chi_V$ is not contained in the parameter of $\pi_0'$, it splits into exactly two irreducible (and non-isomorphic) representations. The following lemma (with $l=1$) thus applies to our situation in this case:
\begin{Lem}
\label{lem_resolve}
Let $\pi_0'$ be an irreducible tempered representation such that $l(\pi_0') = l-2$, whose parameter does not contain $\chi_VS_l$ and (if $l>1$) contains $\chi_VS_{l-2}$ with odd multiplicity. By Lemma \ref{lemma:goldberg}, for any integer $h>0$ the representation $(\chi_V\textnormal{St}_l,h) \rtimes \pi_0'$ is completely reducible, of length two.

If $l > 1$, the two irreducible constituents of this representation differ by their first occurrence indices by Theorem 4.1 of \cite{atobe2017local}. To be precise, we denote by $(V_m)$ the tower on which $\theta_{l-2}(\pi_0') \neq 0$ and let $(V_m')$ the other tower. We may then denote the irreducible constituents of $(\chi_V\textnormal{St}_l,h) \rtimes \pi_0'$ by $\pi_1$ and $\pi_2$ so that $\theta_l(\pi_1) \neq 0$ and $\theta_{l-2}(\pi_2) \neq 0$ are their first occurrences on the tower $(V_m)$.

If $l=1$, we choose one tower $(V_m)$ from the fixed pair of towers (this time the choice is arbitrary) and let the other be $(V_m')$. The first occurrence index (as defined in \S \ref{subs:first_occurrence}) for both constituents is $1$. However, the going-down tower is not the same for $\pi_1$ and $\pi_2$---see Theorem 4.1 (2) in \cite{atobe2017local}. Therefore, we may again denote the irreducible constituents by $\pi_1$ and $\pi_2$ so that $\theta_{1}(\pi_1) \neq 0$ and $\theta_{-1}(\pi_2) \neq 0$ are their first occurrences on the tower $(V_m)$.

Denote by $A$ and $B$ the Langlands quotients of the form
\begin{gather}
L(\chi_W|\cdot|^\frac{l+1}{2}; \tau_1) \label{eq_optionA}\tag{A}\\
L(\chi_W\textnormal{St}_{l+1}\nu^\frac{1}{2}; \tau_2) \label{eq_optionB}\tag{B}
\end{gather}
where $\tau_1$ and $\tau_2$ are some tempered representations. We consider the full lifts $\Theta_{-2-l}(\pi_1)$ and $\Theta_{-2-l}(\pi_2)$ to the tower $(V_m')$.
\begin{enumerate}[(a)]
\item The representation $\Theta_{-2-l}(\pi_1)$ has only one irreducible subquotient of the form \ref{eq_optionB}, namely $\theta_{-2-l}(\pi_1)$. It does not contain any irreducible subquotients of the form \ref{eq_optionA}.
\item The representation $\Theta_{-2-l}(\pi_2)$ has only one irreducible subquotient of the form \ref{eq_optionA}, namely $\theta_{-2-l}(\pi_2)$. It does not contain any irreducible subquotients of the form \ref{eq_optionB}.
\end{enumerate}
\end{Lem}

\begin{proof}
Notice that
\begin{align*}
\Hom_{G_{n}}(\omega_{m,n},({\chi_V\textnormal{St}_l},h) \rtimes \pi_0')_\infty &= \Hom_{G_{n}}(\omega_{m,n}, \pi_1)_\infty \oplus \Hom_{G_{n}}(\omega_{m_0,n_0}, \pi_2)_\infty\\
&= \Theta_{-l-2}(\pi_1)^\vee \oplus \Theta_{-l-2}(\pi_2)^\vee.
\end{align*}
Here $n$ is defined by $\pi_1,\pi_2\in G(W_{n})$, $m = n + \epsilon + l +2$ and $\omega_{m,n}$ is the corresponding Weil representation. A standard application of Kudla's filtration (just like in Corollary \ref{cor:theta_epi}) yields
\[
\Hom_{G_{n}}(\omega_{m,n},({\chi_V\textnormal{St}_l},h) \rtimes \pi_0')_\infty \hookrightarrow ({\chi_W\textnormal{St}_l},h) \rtimes \Theta_{-l-2}(\pi_0')^\vee
\]
which, after taking contragredients, leads to
\[
({\chi_W\textnormal{St}_l},h) \rtimes \Theta_{-l-2}(\pi_0') \twoheadrightarrow 
 \Theta_{-l-2}(\pi_1) \oplus \Theta_{-l-2}(\pi_2).
\]
We are looking for non-tempered subquotients of $\Theta_{-l-2}(\pi_1)$ and $\Theta_{-l-2}(\pi_2)$. We now use Theorem 4.1 of \cite{muic2008structure}, which shows that the only non-tempered subquotient of $\Theta_{-l-2}(\pi_0')$ is $\theta_{-l-2}(\pi_0') = L(\chi_W\allowbreak |\cdot|^{\frac{l+1}{2}} \allowbreak \rtimes \theta_{-l}(\pi_0'))$. We note that, although the theorem is originally stated for discrete series representations, it may be applied here; the same proof works because the parameter of $\pi_0'$ does not contain $\chi_VS_l$. This implies that the non-tempered irreducible subquotients of $\Theta_{-l-2}(\pi_1) \oplus \Theta_{-l-2}(\pi_2)$ are necessarily contained in $({\chi_W\textnormal{St}_l},h)\times \chi_W|\cdot|^\frac{l+1}{2} \rtimes \Theta_{-l}(\pi_0')$, i.e.~in
\[
\Pi=\chi_W|\cdot|^\frac{l+1}{2} \times ({\chi_W\textnormal{St}_l},h) \rtimes \Theta_{-l}(\pi_0').
\]
A standard argument using the tempered support and the irreducibility of first lifts for discrete series (like the one used in Case $1$ of \S $5$) combined with Theorem 4.5 of \cite{atobe2017local} now shows that $\Theta_{-l}(\pi_0')$ is irreducible and tempered. Furthermore, its parameter contains $\chi_WS_l$, so that $({\chi_W\textnormal{St}_l},h) \rtimes \Theta_{-l}(\pi_0')$ is also irreducible and tempered; we denote this representation by $T$.  By Proposition \ref{prop:lifts_temp} we know that $\theta_{-l-2}(\pi_1)$ is of the form \ref{eq_optionB}, whereas $\theta_{-l-2}(\pi_2)$ is of the form \ref{eq_optionA}. To be precise, we have
\begin{align*}
\theta_{-l-2}(\pi_1) = L(\chi_W\textnormal{St}_{l+1}\nu^\frac{1}{2}; ({\chi_W\textnormal{St}_l},h-1) \rtimes \Theta_{-l}(\pi_0'))\\
\theta_{-l-2}(\pi_2) = L(\chi_W|\cdot|^\frac{l+1}{2}; ({\chi_W\textnormal{St}_l},h) \rtimes \Theta_{-l}(\pi_0')).
\end{align*}

Therefore, $\Pi$ contains at least one irreducible subquotient of the form \ref{eq_optionA}, and one of the form \ref{eq_optionB}. The Lemma will be proven if we show that $\Pi$ contains no other irreducible subquotients of the form \ref{eq_optionA} or B. We establish this by proving the following:
\begin{itemize}
\item any subquotient of the form \ref{eq_optionA} appearing in $\Pi$ is necessarily isomorphic to $\theta_{-l-2}(\pi_2)$;
\item any subquotient of the form \ref{eq_optionB} appearing in $\Pi$ is necessarily isomorphic to $\theta_{-l-2}(\pi_1)$;
\item $\theta_{-l-2}(\pi_1)$ and $\theta_{-l-2}(\pi_2)$ appear in $\Pi$ with multiplicity $1$.
\end{itemize}
To prove the first of these claims, let $\pi'$ be a subquotient of $\Pi$ of the form $A$. We then have $\chi_W|\cdot|^\frac{l+1}{2} \rtimes \tau_1 \twoheadrightarrow \pi'$, or, using Lemma \ref{lemma:MVWinv}, $\pi' \hookrightarrow \chi_W|\cdot|^{-\frac{l+1}{2}} \rtimes \tau_1$. Frobenius reciprocity now implies
$R_{P_1}(\pi') \twoheadrightarrow \chi_W|\cdot|^{-\frac{l+1}{2}} \otimes \tau_1$.
In other words, $R_{P_1}(\pi')$ contains a (sub)quotient of the form $\chi_W|\cdot|^{-\frac{l+1}{2}} \otimes \tau_1$. On the other hand, we may compute $R_{P_1}(\Pi)$ using formula \eqref{eq_Tadic}, $\mu^*( \chi_W|\cdot|^\frac{l+1}{2} \rtimes T) = M^*(\chi_W|\cdot|^\frac{l+1}{2}) \rtimes \mu^*(T)$. As $T$ is tempered, Casselman's criterion shows that $R_{P_1}(T)_{\chi_W|\cdot|^{-\frac{l+1}{2}}} = 0$. From here, it follows that the only irreducible subquotient of the form $\chi_W|\cdot|^{-\frac{l+1}{2}} \otimes \tau_1$ in the semisimplification of $R_{P_1}(\Pi)$ is ${\chi_W|\cdot|^{-\frac{l+1}{2}}}\otimes T$. Therefore, the Langlands quotient of $\chi_W|\cdot| \rtimes T$ is the only subquotient of $\Pi$ of the form \ref{eq_optionA}.

The second claim, about subquotients of the form \ref{eq_optionB}, is proven using a similar, albeit more complicated argument. We begin in the same manner: if $\pi'$ is any irreducible subquotient of $\Pi$ of the form \ref{eq_optionB}, then $R_{P_{l+1}}(\pi')$ contains a (sub)quotient of the form $\chi_W\textnormal{St}_{l+1}\nu^{-\frac{1}{2}} \otimes \tau_2$. We now use the the formulae \eqref{eq_Tadic} and \eqref{eq_M} to look for subquotients of this form in $R_{P_{l+1}}(\Pi)$ and show that in this case $\tau_2 = ({\chi_W\textnormal{St}_l},h-1)  \rtimes \Theta_{-l}(\pi_0')$ is necessary. The calculation needed here is somewhat arduous; we omit the details because we give a detailed description of a similar Jacquet module argument in the proof of the third claim below. We restrict ourselves to noting that, apart from Casselman's criterion, in this case we also use the fact that $R_{P_l}( \Theta_{-l}(\pi_0'))$ does not contain an irreducible subquotient of the form $\chi_W\textnormal{St}_l \otimes \tau'$ (for some irreducible tempered representation $\tau'$). This follows from the fact that $\chi_WS_l$ appears in the parameter of $\Theta_{-l}(\pi_0')$ with multiplicity $1$, whereas the existence of a subquotient of the form $\chi_W\textnormal{St}_l \otimes \tau'$ would imply multiplicity greater than or equal to $2$.

Finally, to prove the third claim, we count the (irreducible) subquotients of the form $({\chi_W\textnormal{St}_l},h)  \otimes \pi''$, where $\pi''$ is irreducible and non-tempered, in $R_{P_{lh}}(\Pi)$. We claim that there are exactly $2^h$ such subquotients in $R_{P_{lh}}(\Pi)$. To prove this, we use formulae \eqref{eq_Tadic} and \eqref{eq_M} once again. They show that the semisimplification of $R_{P_{lh}}(\Pi)$ is equal to the (semisimplification of the) following sum:
\begin{align*}
\label{eq_semi}\tag{\textasteriskcentered}
\sum_{\substack{i_1,\dotsc,i_h \\ \frac{-1-l}{2} \leq i_k \leq \frac{l-1}{2}}} \ \sum_{\substack{j_1,\dotsc,j_h \\ i_k \leq {j_k} \leq \frac{l-1}{2}, \forall k}}&\prod_{k=1}^{h} \chi_W\langle |\cdot|^\frac{l-1}{2},|\cdot|^{-i} \rangle \times \chi_W\langle |\cdot|^\frac{l-1}{2}, |\cdot|^{j+1}\rangle \otimes \chi_W\langle |\cdot|^{j}, |\cdot|^{i+1}\rangle\\
& \times (\chi_W|\cdot|^\frac{l+1}{2} \otimes \mathbb{1} + \chi_W|\cdot|^{-\frac{l+1}{2}} \otimes \mathbb{1} + \mathbb{1}\otimes \chi_W|\cdot|^\frac{l+1}{2})\\
&\times (\delta \otimes \tau),
\end{align*}
where $\delta \otimes \tau$ goes over irreducible subquotients in $\mu^*(\Theta_{-l}(\pi_0'))$. Recall that, for the above formula, we interpret the segment notation in the first line so that $\langle |\cdot|^\frac{l-1}{2},|\cdot|^{\frac{l+1}{2}} \rangle = \langle |\cdot|^{i},\allowbreak |\cdot|^{i+1} \rangle = \mathbb{1} \in \Irr(\GL_0(\F))$. Furthermore, to help readability, the index $k$ is dropped from $i_k$ and $j_k$ in the first line.

We are looking for subquotients of the form $({\chi_W\textnormal{St}_l},h)  \otimes \pi''$ in the above sum. We concentrate on the $\GL$ part of such subquotients, i.e.~$({\chi_W\textnormal{St}_l},h)$. First, notice that $\chi_W|\cdot|^{\pm \frac{l+1}{2}}$ does not appear in the cuspidal support of $({\chi_W\textnormal{St}_l},h)$. Therefore, such subquotients only appear in the part of \eqref{eq_semi} which corresponds to the summand $\mathbb{1}\otimes \chi_W|\cdot|^\frac{l+1}{2}$ in the second row.

We now determine which summands of \eqref{eq_semi} possess a subquotient of the form $({\chi_W\textnormal{St}_l},h)  \otimes \pi''$. Notice that in any such summand the cuspidal representation $\chi_W|\cdot|^\frac{1-l}{2}$ appears in the $\GL$ part exactly $h$ times. A key observation is that in any such summand, $\chi_W|\cdot|^\frac{1-l}{2}$ cannot appear in the cuspidal support of $\delta$. Indeed, if some $\delta$ had $\chi_W|\cdot|^\frac{1-l}{2}$ in its cuspidal support, it would then follow that $R_P(\Theta_{-l}(\pi_0'))$ (for a suitable choice of standard parabolic $P$) contains a subquotient of the form $\chi_W\langle |\cdot|^y, |\cdot|^x\rangle \otimes \tau'$ (for some irreducible representation $\tau$), with $y-x \in \mathbb{Z}_{\geq 0}$ and $\frac{1-l}{2} \in \{x,x+1,\dotsc,y\}$. However, since we are only considering summands where the $\GL$ part is equal to $({\chi_W\textnormal{St}_l},h)$, $y$ cannot be greater than $\frac{l-1}{2}$. This shows that $x+y \leq 0$. If $x+y < 0$, we arrive at a contradiction with the Casselman's temperedness criterion. On the other hand, $x+y = 0$ implies that $x = -y = \frac{1-l}{2}$, i.e.~that $\chi_W\textnormal{St}_l \otimes \tau'$ is a subquotient of $R_{P_l}(\Theta_{-l}(\pi_0'))$. However, this in turn implies that $\chi_WS_l$ appears with multiplicity at least $2$ in the parameter of $\Theta_{-l}(\pi_0')$, and this contradicts the fact that the parameter of $\Theta_{-l}(\pi_0')$ contains $\chi_WS_l$ with multiplicity one. We have thus proven that $\delta$ cannot contribute to the total number of representations $\chi_W|\cdot|^\frac{1-l}{2}$ in the cuspidal support of the $\GL$ part of a summand which contains $({\chi_W\textnormal{St}_l},h) \otimes \pi''$.

In other words, all the representations $\chi_W|\cdot|^\frac{1-l}{2}$ in the cuspidal support of $({\chi_W\textnormal{St}_l},h)$ come from the first row of \eqref{eq_semi}. Recall that we need to have exactly $h$ of them. One verifies easily that $\chi_W|\cdot|^\frac{1-l}{2}$ appears only in summands where $i_k = j_k \in \{ -\frac{l+1}{2},\frac{l-1}{2} \}$ for all $k \in \{1,\dotsc,h\}$. Thus, for each index $k = 1, \dotsc, h$, we have two choices for the value of $i_k=j_k$. Therefore, there are $2^h$ such summands in \eqref{eq_semi}; they are all of the form
\[
({\chi_W\textnormal{St}_l},h) \otimes \chi_W|\cdot|^\frac{l+1}{2} \rtimes \Theta_{-l}(\pi_0').
\]
Every such summand obviously contains at least one irreducible subquotient of the form $({\chi_W\textnormal{St}_l},h) \otimes \pi''$ with $\pi''$ non-tempered, namely $({\chi_W\textnormal{St}_l},h) \otimes L(\chi_W|\cdot|^\frac{l+1}{2} ;\Theta_{-l}(\pi_0'))$. To prove that there are exactly $2^h$ such subquotients in $R_{P_{hl}}(\Pi)$, it remains to prove that  $L(\chi_W \allowbreak |\cdot|^\frac{l+1}{2} ;\Theta_{-l}(\pi_0'))$ is the only non-tempered irreducible subquotient of $\chi_W|\cdot|^\frac{l+1}{2} \rtimes \Theta_{-l}(\pi_0')$. We leave the proof of this simple fact to the reader; let us mention that the proof of this fact again combines a Jacquet module analysis with Casselman's criterion and the fact that $\chi_WS_l$ appears with multiplicity $1$ in the parameter of $\Theta_{-l}(\pi_0')$.

We are now ready to finish the proof of the third claim (and thus of the Lemma), i.e.~of the fact that $\theta_{-l-2}(\pi_1)$ and $\theta_{-l-2}(\pi_2)$ appear in $\Pi$ with multiplicity one. To prove this, it suffices to prove that both $R_{P_{lh}}(\theta_{-l-2}(\pi_1))$ and $R_{P_{lh}}(\theta_{-l-2}(\pi_2))$ contain exactly $2^{h-1}$ irreducible subquotients of the form $({\chi_W\textnormal{St}_l},h) \otimes \pi''$ with $\pi''$ non-tempered. We omit the proof of this fact because it follows the same steps used above to determine the number of such subquotients in $R_{P_{lh}}(\Pi)$. We point out that, when repeating the computation in case of $\pi_2$, it is helpful to notice that $\theta_{-l-2}(\pi_2)$ is actually a quotient of $\chi_WL(\chi_W|\cdot|^\frac{l+1}{2}, \textnormal{St}_l) \times ({\chi_W\textnormal{St}_l},h-1) \rtimes \Theta_{-l}(\pi_0')$. This concludes the proof of the lemma.
\end{proof}

Let us return to our analysis of the lifts in case 1b. We now apply Lemma \ref{lem_resolve} with $l=1$ to the discussion preceding the Lemma. If we are lifting $\pi_0$ to the going-up tower, Lemma \ref{lem_resolve} (a) with $\pi_0 = \pi_1$ shows that (i) cannot hold, because $\Theta_{-3}(\pi_0)$ contains no subquotient with standard module of the form $\chi_W|\cdot|^1 \rtimes \tau_1$. Therefore (ii) holds, and moreover, Lemma \ref{lem_resolve} (b) shows that in this case $L(\chi_W\textnormal{St}_2\nu^\frac{1}{2}; \tau_2)$ is in fact equal to $\theta_{-3}(\pi_0)$. Comparing this with Proposition \ref{prop:lifts_temp}, we see that the irreducible subquotient of $\Theta_{-l}(\pi_0)$ whose standard module is described by (ii) is equal to $\theta_{-l}(\pi_0)$.

Therefore, on the going-up tower we have
\[
\chi_W\delta_r\nu^{s_r} \times \dotsm \times \chi_W\delta_1\nu^{s_1} \rtimes  \theta_{-l}(\pi_{0}) \twoheadrightarrow \theta_{-l}(\pi).
\]
Furthermore, the proof of Lemma \ref{lemma:unique_quotient} (which we have already used to arrive at (i) and (ii)) shows that the standard module of $\theta_{-l}(\pi)$ is obtained by adding $\chi_W\delta_r\nu^{s_r}, \dotsc,\chi_W\delta_1\nu^{s_1}$ to the factors appearing in the standard module of $\theta_{-l}(\pi_0)$ (and sorting decreasingly). We arrive at the analogous conclusion about the lifts to the going-down tower by using Lemma \ref{lem_resolve} with $\pi_0 = \pi_2$ to show that, on the going-down tower, (i) must hold with $\tau_1 = \theta_{-1}(\pi_0)$.

This concludes case 1b.

\bigskip

\noindent\underline{\textbf{Case 2:}{ $m_\phi(\chi_V)$ is odd, going-down tower}}\\
This case is much simpler than case 1b. Just as in the first two cases, we have $\chi_W\langle|\cdot|^1, \allowbreak |\cdot|^e\rangle \rtimes \theta_{-1}(\pi_{00}) \twoheadrightarrow \theta_{-l}(\pi_{00})$, so \eqref{eq:epi2} leads to
\[
\chi_W\delta_r\nu^{s_r} \times \dotsm \times \chi_W\delta_1\nu^{s_1} \times \chi_W\Delta \times \chi_W\langle|\cdot|^1,|\cdot|^e\rangle \rtimes \theta_{-1}(\pi_{00}) \twoheadrightarrow \sigma.
\]
The complicated part in the previous case was handling the exceptional cases in which $\chi_W\langle|\cdot|^1,|\cdot|^e\rangle$ could not switch places with all the representations which define $\chi_W\Delta$ (namely, $\chi_W$).

The difference here is the following: since the multiplicity of $\chi_V$ in $\phi$ is now odd, one copy of $\chi_V$ must also appear in the parameter of $\pi_{00}$. Using \cite[Theorem 4.3 (2)]{atobe2017local}, we see that in this case $\theta_{-1}(\pi_{00})$ is not a discrete series representation, but a tempered representation with $\chi_W$ occurring twice in its parameter. This means that there is a discrete series representation $\sigma_{00}$ such that $\theta_{-1}(\pi_{00}) \hookrightarrow \chi_W \rtimes \sigma_{00}$ (moreover, $\theta_{-1}(\pi_{00})$ splits off).

We claim that $\theta_{-l}(\pi_{00})$, as the (unique) quotient of $\chi_W\langle|\cdot|^1,|\cdot|^e\rangle \rtimes \theta_{-1}(\pi_{00})$, must also be a quotient of
\[
\chi_W\langle|\cdot|^0,|\cdot|^e\rangle \rtimes \sigma_{00}.
\]
Otherwise, Lemma \ref{rem:swap2} would imply that $\theta_{-l}(\pi_{00})$ is a quotient of $\chi_W|\cdot|^e \times \dotsm \times \chi_W|\cdot|^2 \times \chi_W\textnormal{St}_2\nu^{\frac{1}{2}} \rtimes \sigma_{00}$, and this is not possible because of the uniqueness of the standard module of $\theta_{-l}(\pi_{00})$. Therefore, \eqref{eq:epi2} in fact leads to
\[
\chi_W\delta_r\nu^{s_r} \times \dotsm \times \chi_W\delta_1\nu^{s_1} \times \chi_W\Delta \times \chi_W\langle|\cdot|^0,|\cdot|^e\rangle \rtimes \sigma_{00} \twoheadrightarrow \sigma,
\]
which simplifies things considerably: in contrast to $\chi_W\langle|\cdot|^1,|\cdot|^e\rangle$, $\chi_W\langle|\cdot|^0,|\cdot|^e\rangle$ can switch places with all the representations appearing in the definition of $\chi_W\Delta$, according to Lemma \ref{lemma:irr_zel}. We thus get
\[
\chi_W\delta_r\nu^{s_r} \times \dotsm \times \chi_W\delta_1\nu^{s_1} \times \chi_W\langle|\cdot|^0,|\cdot|^e\rangle \times \chi_W\Delta \rtimes \sigma_{00} \twoheadrightarrow \sigma.
\]
Using $\chi_W\langle|\cdot|^1,|\cdot|^e\rangle \times \chi_W \twoheadrightarrow \chi_W\langle|\cdot|^0,|\cdot|^e\rangle$, this leads to
\[
\chi_W\delta_r\nu^{s_r} \times \dotsm \times \chi_W\delta_1\nu^{s_1} \times \chi_W\langle|\cdot|^1,|\cdot|^e\rangle \times \chi_W \times \chi_W\Delta \rtimes \sigma_{00} \twoheadrightarrow \sigma.
\]
We can now deduce that there is an irreducible (and obviously tempered) subquotient $\tau$ of $\chi_W \times \chi_W\Delta \rtimes \sigma_{00}$ such that
\[
\chi_W\delta_r\nu^{s_r} \times \dotsm \times \chi_W\delta_1\nu^{s_1} \times \chi_W\langle|\cdot|^1,|\cdot|^e\rangle \rtimes \tau \twoheadrightarrow \sigma.
\]
From here, we can repeat the arguments of Case 1a word for word; this leads to the same conclusion regarding the standard module of $\theta_{-l}(\pi)$.

\bigskip

\noindent\underline{\textbf{Case 3:}{ $m_\phi(\chi_V)$ is odd, going-up tower}}\\
In this case we have $\chi_W\langle|\cdot|^2,|\cdot|^e\rangle \rtimes \theta_{-3}(\pi_{00}) \twoheadrightarrow \theta_{-l}(\pi_{00})$ so that \eqref{eq:epi2} leads to
\[
\chi_W\delta_r\nu^{s_r} \times \dotsm \times \chi_W\delta_1\nu^{s_1} \times \chi_W\Delta \times \chi_W\langle|\cdot|^2,|\cdot|^e\rangle  \rtimes \theta_{-3}(\pi_{00}) \twoheadrightarrow \sigma.
\]
According to Lemma \ref{lemma:irr_zel}, $\chi_W\langle|\cdot|^2,|\cdot|^e\rangle $ can switch places with all the representations which define $\chi_W\Delta$, except $\chi_W\textnormal{St}_3 = \chi_W\langle|\cdot|^1,|\cdot|^{-1}\rangle$. Thus we consider two subcases again, depending on the parity of $m_\phi(\chi_VS_3)$.

\bigskip
\noindent\textbf{Case 3a:}{ $m_\phi(\chi_VS_3)=2h+1$}\\
The approach to this Case is the same as the one we used for Case 2, mutatis mutandis. In this case the parameter of $\pi_{00}$ contains $\chi_VS_3$ so Theorem 4.5 (1) of \cite{atobe2017local} shows that the multiplicity of $\chi_WS_3$ in the parameter of $\theta_{-3}(\pi_{00})$ equals $2$. Therefore, $\theta_{-3}(\pi_{00})$ is not in discrete series (but is tempered), and there is a discrete series representation $\sigma_{00}$ such that
\[
\chi_W\textnormal{St}_3 \rtimes \sigma_{00} \twoheadrightarrow \theta_{-3}(\pi_{00}).
\]
We now use an observation similar to the one we used in Case 2. Lemma \ref{rem:swap2} shows that
\[
\chi_WL \rtimes \sigma_{00} \twoheadrightarrow \theta_{-l}(\pi_{00}),
\]
where $L = L(|\cdot|^e\times \dotsm \times |\cdot|^2 \times \textnormal{St}_3)$. Another option is
\[
\chi_WL' \rtimes \sigma_{00} \twoheadrightarrow \theta_{-l}(\pi_{00}),
\]
with $L' = L(|\cdot|^e\times \dotsm \times|\cdot|^3 \times \textnormal{St}_4\nu^\frac{1}{2})$, but this contradicts what we know about the standard module of $\theta_{-l}(\pi_{00})$. Thus, starting once again from \eqref{eq:epi2}, we can write
\[
\chi_W\delta_r\nu^{s_r} \times \dotsm \times \chi_W\delta_1\nu^{s_1} \times \chi_W\Delta \times \chi_WL \rtimes \sigma_{00} \twoheadrightarrow \sigma.
\]
Furthermore, we may write $\Delta = \Delta' \times (\textnormal{St}_3,h)$, just like in Case 1b. Lemma \ref{lemma_specific_reducibility} now shows that $\chi_WL$ can be swapped with $\chi_W(\textnormal{St}_3,h)$. Therefore, we have
\[
\chi_W\delta_r\nu^{s_r} \times \dotsm \times \chi_W\delta_1\nu^{s_1} \times \chi_W\Delta' \times \chi_WL \times \chi_W(\textnormal{St}_3,h) \rtimes \sigma_{00} \twoheadrightarrow \sigma.
\]
Finally, as $\langle|\cdot|^2,|\cdot|^e\rangle  \times \textnormal{St}_3 \twoheadrightarrow L$, and $\langle|\cdot|^2,|\cdot|^e\rangle $ can switch places with all the representations which define $\Delta'$, we arrive at
\[
\chi_W\delta_r\nu^{s_r} \times \dotsm \times \chi_W\delta_1\nu^{s_1} \times \chi_W\langle|\cdot|^2,|\cdot|^e\rangle \times \chi_W\Delta' \times \chi_W(\textnormal{St}_3,h+1) \rtimes \sigma_{00} \twoheadrightarrow \sigma.
\]
Just like in Case 2, we conclude that there is an irreducible (tempered) subquotient $\tau$ of $ \chi_W\Delta' \times \chi_W(\textnormal{St}_3,h+1) \rtimes \sigma_{00}$ such that
\[
\chi_W\delta_r\nu^{s_r} \times \dotsm \times \chi_W\delta_1\nu^{s_1} \times \chi_W\langle|\cdot|^2,|\cdot|^e\rangle  \times \tau \twoheadrightarrow \sigma.
\]
Note that in Case 3, $l(\pi_0)=1$ and $m_\phi(\chi_V)$ is odd. Lemma \ref{lem:no_ex} thus shows that none of the factors $\delta_r\nu^{s_r}$ are equal to $|\cdot|^1$ or $\textnormal{St}_2\nu^\frac{1}{2}$. Therefore, we may apply Lemma \ref{lemma:unique_quotient}. Once again, we repeat the arguments of Case 1a: Lemma \ref{lemma:unique_quotient} applied to the above epimorphism shows that we have 
\[
\chi_W\Pi \rtimes \tau' \twoheadrightarrow \theta_{-l}(\pi).
\]
for $\tau' = L(\chi_W|\cdot|^e, \dotsc, \chi_W|\cdot|^2; \tau)$. By Lemma \ref{lemma:really_subq} we may assume that $\tau'$ is a subquotient of $\Theta_{-l}(\pi_0)$; we then use Lemma \ref{lem_descend} like in Lemma \ref{lemma:really_small_theta} to show that $\tau' = \theta_{-l}(\pi_0)$; note that $\Theta_{-3}(\pi_0)$ is irreducible (cf.~Case 2 in \S $5$). We thus arrive at the conclusion:
\[
\chi_W\delta_r\nu^{s_r} \times \dotsm \times \chi_W\delta_1\nu^{s_1} \rtimes  \theta_{-l}(\pi_{0}) \twoheadrightarrow \theta_{-l}(\pi)
\]
and $\theta_{-l}(\pi) = L(\chi_W\delta_r\nu^{s_r}, \dotsc, \chi_W\delta_1\nu^{s_1}, |\cdot|^\frac{l-1}{2},\dotsc,|\cdot|^2,\theta_{-3}(\pi_0))$.

\medskip

\noindent\textbf{Case 3b:}{ $m_\phi(\chi_VS_3)=2h$}\\
\noindent In this case $\theta_{-3}(\pi_{00})$ is a discrete series representation such that $\chi_W\langle|\cdot|^2,|\cdot|^e\rangle \rtimes \theta_{-3}(\pi_{00}) \twoheadrightarrow \theta_{-l}(\pi_{00})$. Our approach in this case is similar to Case 1b.

Again we group $\chi_W\Delta =  (\chi_W\textnormal{St}_3,h) \times \chi_W\Delta'$ and notice that $\chi_W\langle|\cdot|^2,|\cdot|^e\rangle$ can switch places with $\chi_W\Delta'$, whereas switching places with $\chi_W\textnormal{St}_3$ leads to two different possibilities branching from \eqref{eq:epi2} (similar to the situation in Case 1b):
\begin{enumerate}[(i)]
\item there is an irreducible tempered representation $\tau_1$ such that
\[
\chi_W\delta_r\nu^{s_r} \times \dotsm \times \chi_W\delta_1\nu^{s_1} \times  \chi_W\langle|\cdot|^2,|\cdot|^e\rangle  \rtimes \tau_1 \twoheadrightarrow \sigma;
\]
\item  there is an irreducible tempered representation $\tau_2$ such that
\[
\chi_W\delta_r\nu^{s_r} \times \dotsm \times \chi_W\delta_1\nu^{s_1} \times   \chi_WL' \rtimes \tau_2 \twoheadrightarrow \sigma,
\]
\end{enumerate}
where $L'$ denotes $L(|\cdot|^e\times \dotsm \times|\cdot|^3 \times \textnormal{St}_4\nu^\frac{1}{2})$, i.e.~the unique irreducible quotient of $\langle|\cdot|^3,|\cdot|^e\rangle \times \textnormal{St}_4\nu^\frac{1}{2}$.

We now show that (ii) is not possible. First, we show that the left-hand side of (ii) has a unique irreducible quotient, since it is itself a quotient of a standard representation. This is shown using Lemma \ref{lemma:unique_quotient}; we just have to ensure that there are no representations among $\chi_W\delta_r\nu^{s_r}, \dotsc, \chi_W\delta_1\nu^{s_1}$ which violate condition \eqref{eq:tricky1} of Lemma \ref{lem:tricky_condition}. The representations causing the problems are those defined by a segment which ends in $\chi_W|\cdot|^2$. These are $\chi_W|\cdot|^2$, $\chi_W\langle|\cdot|^2,|\cdot|^1\rangle $, $\chi_W\langle|\cdot|^2,|\cdot|^0\rangle $ and $\chi_W\langle|\cdot|^2,|\cdot|^{-1}\rangle  = \chi_W\textnormal{St}_4\nu^\frac{1}{2}$.

Notice that for (ii) to be possible, $\chi_V\textnormal{St}_3$ needs to appear in the tempered support of $\pi_0$. But then the results of \cite{muic2005reducibility} imply that the standard module reduces whenever it contains the factor $\chi_W|\cdot|^2$, $\chi_W\langle|\cdot|^2,|\cdot|^1\rangle $ or $\chi_W\langle|\cdot|^2,|\cdot|^0\rangle $. Since the standard module of $\pi$ is irreducible, we may therefore assume that none of these exceptions appear. Thus, the only possible exception is $\chi_W\textnormal{St}_4\nu^\frac{1}{2}$. On the other hand, Lemma \ref{lemma_specific_reducibility} shows that $\textnormal{St}_4\nu^\frac{1}{2} \times L'$ is irreducible. This means that we can swap $\chi_W\textnormal{St}_4\nu^\frac{1}{2}$ and $\chi_WL'$ in (ii), getting
\[
\chi_W\delta_r\nu^{s_r} \times \dotsm \times \chi_W\delta_2\nu^{s_2} \times   \chi_W\langle|\cdot|^3,|\cdot|^e\rangle \times \chi_W\textnormal{St}_4\nu^\frac{1}{2}\times \chi_W\textnormal{St}_4\nu^\frac{1}{2}  \rtimes \tau_2.
\]
This enables us to apply Lemma \ref{lemma:unique_quotient} and show that the above representation has a unique irreducible quotient. We have now arrived at the following conclusion: if there were really an epimorphism like in (ii), then the irreducible subquotient of $\chi_WL' \rtimes \tau_2$ which participates in it would in fact be the unique quotient, i.e.~$L(|\cdot|^e\times \dotsm \times|\cdot|^3 \times \textnormal{St}_4\nu^\frac{1}{2}\rtimes \tau_2)$. As in the previous cases, Lemma \ref{lemma:really_subq} shows that this representation is also a subquotient of $\Theta_{-l}(\pi_0)$. An inductive application of Lemma \ref{lem_descend} (like in Lemma \ref{lemma:really_small_theta}) then implies that $L(\textnormal{St}_4\nu^\frac{1}{2}\rtimes \tau_2)$ is a subquotient of $\Theta_{-5}(\pi_0)$.

We now use Lemma \ref{lem_resolve} (with $l=3$) to show that this is not possible. In the notation of the lemma, $\pi_0$ corresponds to $\pi_2$, and we are lifting it to the going-up tower $(V_m')$. Part (b) of the lemma now shows that $\Theta_{-5}(\pi_0)$ contains no subquotients of the form $L(\textnormal{St}_4\nu^\frac{1}{2}\rtimes \tau_2)$.

We have thus shown that (i) always holds in Case 3b. Since the factors $\chi_W|\cdot|$ or $\chi_W\textnormal{St}_2\nu^\frac{1}{2}$ cannot appear in the standard module (by Lemma \ref{lem:no_ex}), there are no exceptions to Lemma \ref{lem:tricky_condition} and we may proceed with the proof just like in Case 1a, using the arguments of Lemmas \ref{lemma:unique_quotient}, \ref{lemma:really_subq} and \ref{lem_descend}.

\bigskip

\noindent This completes the analysis of the cases obtained by considering all the different possibilities for $\pi_0$. Together with the lifts determined in Section \ref{sec:lifts} it provides a complete description of all the lifts we have considered. The results are summarized in Theorem \ref{thm:lifts}.
\section{Unitary representations}
\label{sec:unitary}

A direct consequence of our results is a method for constructing a series of unitary representations of both $\Ort(V_m)$ and $\Sp(W_n)$.

To be more specific, the structure of the generic unitary dual is known by the work of Lapid, Mui\'{c} and Tadi\'{c} \cite{lapid2004generic}, whereas Theorem \ref{thm:lifts} provides an explicit description of the lifts of $\chi$-generic representations. On the other hand, the results of J.-S.~Li \cite{li1989singular} imply that the lifts of unitary representations in the stable range remain unitary.

Therefore, taking any $\chi$-generic unitary irreducible representation of $\Ort(V_m)$ or $\Sp(W_n)$ we obtain a sequence of (non-generic) unitary representations by looking at its theta lifts in the stable range. We thus arrive at the following result:
\begin{Prop}
Let $\tau$ be a tempered unitary $\chi$-generic representation of $G(W_n)$. Fix a Witt tower $\mathcal{V} = (V_m)$ and let $m_0$ denote the dimension of the anisotropic space in $\mathcal{V}$. 

Let $m_1 = \min\{m \in 2\mathbb{Z}: m > n + \epsilon, \theta(\tau,m) \neq 0\}$. Choose
\begin{itemize}
\item any $r \geq n$ and let $m_2 = m_0 + 2r$ (so that $\theta(\tau,m_2)$ is in the stable range); and
\item any sequence of $\delta_i \in \Irr_{\text{disc}}\GL_{d_i}(\F)$, $i=1,\dotsc,r$, and $s_r \geqslant \dotsb \geqslant s_1 > 0$ such that  $\chi_V\delta_r\nu^{s_r} \times \dotsm \times \chi_V\delta_1\nu^{s_1} \rtimes \tau$ satisfies conditions of Theorem 1.1, \cite{lapid2004generic}.
\end{itemize}
Finally, set $l_i = n - m_i + \epsilon$ for $i=1,2$. Then
\[
L(\chi_W\delta_r\nu^{s_r}, \dotsc, \chi_W\delta_1\nu^{s_1}, \chi_W|\cdot|^\frac{-1-l_2}{2}, \dotsc, \chi_W|\cdot|^\frac{1-l_1}{2}; \theta(\tau,m_1))
\]
is unitarizable.
\end{Prop}

As the structure of the unitary dual of the classical groups is still largely unknown (especially the parts which are not local components of square-integrable automorphic forms classified in \cite{arthur2013endoscopic}), this result offers  a potentially useful insight.

\end{document}